\documentclass[11pt]{article}
\usepackage[margin=0.95in]{geometry}

\usepackage{lipsum}
\usepackage{amsmath, amssymb, mathtools}
\usepackage{pifont}
\usepackage{graphicx}
\usepackage{epstopdf}
\usepackage{enumitem}
\usepackage{algorithm}
\usepackage{algorithmicx,algpseudocode}
\ifpdf
\DeclareGraphicsExtensions{.eps,.pdf,.png,.jpg}
\else
\DeclareGraphicsExtensions{.eps}
\fi

\usepackage[utf8]{inputenc} % allow utf-8 input
\usepackage{url}            % simple URL typesetting
\usepackage{booktabs}       % professional-quality tables
\usepackage{nicefrac}       % compact symbols for 1/2, etc.
\usepackage{microtype}      % microtypography
\usepackage{xcolor}         % colors

\usepackage{amsthm}

\usepackage{multirow}
\usepackage{caption} 
\usepackage{hyperref}
\hypersetup{colorlinks={true},linkcolor={magenta},citecolor={blue}}
\usepackage[capitalise]{cleveref}
\crefname{equation}{}{}
\Crefname{equation}{Equation}{Equations}
\crefname{figure}{Fig.}{Figs.}

\usepackage{graphicx}
\usepackage{bmpsize}
\usepackage{tikz}
\usepackage{subfig}

\usepackage{xparse}
\usepackage{xspace}

\usepackage[normalem]{ulem}
\usepackage{cite}
\usepackage{bm}
\usepackage{tcolorbox}
\usepackage{enumitem}
\usepackage{epstopdf}
\ifpdf
\DeclareGraphicsExtensions{.eps,.pdf,.png,.jpg}
\else
\DeclareGraphicsExtensions{.eps}
\fi

\newtheorem{theorem}{Theorem}%  meant for continuous numbers
\newtheorem{proposition}[theorem]{Proposition}%
\newtheorem{corollary}[theorem]{Corollary}
\newtheorem{lemma}[theorem]{Lemma}
\newtheorem{fact}[theorem]{Fact}

\numberwithin{equation}{section}

\newcommand{\Br}{\mathbb{R}}

\newcommand{\tr}{{\rm tr}}

\newcommand{\diag}{{\rm diag}}

\newcommand{\St}{{\rm St}}
\newcommand{\grad}{\textnormal{grad}\,}

\newcommand{\br}{\mathbb{R}}

\newcommand{\vvx}{\bm{x}}
\newcommand{\vvy}{\bm{y}}
\newcommand{\vvz}{\bm{z}}
\newcommand{\ba}{\begin{array}}
	\newcommand{\ea}{\end{array}}

\newcommand{\PCal}{\mathcal{P}}

\newcommand{\EE}{{\mathbb{E}}}

\newcommand{\NCal}{\mathcal{N}}

\def\bI{\bm{I}}
\def\bA{\bm{A}}
\def\bM{\bm{M}}
\def\bQ{\bm{Q}}
\def\bU{\bm{U}}
\def\bV{\bm{V}}
\def\bX{\bm{X}}
\def\bZ{\bm{Z}}
\def\bPi{\bm{\Pi}}

\def\ba{\bm{a}}
\def\bq{\bm{q}}

\def\bx{\bm{x}}

\title{On the Estimation Performance of Generalized Power Method for Heteroscedastic Probabilistic PCA}

\author{Jinxin Wang\thanks{Department of Systems Engineering and Engineering Management, The Chinese University of Hong Kong.}
	\and Chonghe Jiang\thanks{Department of Systems Engineering and Engineering Management, The Chinese University of Hong Kong.
	}
	\and Huikang Liu\thanks{Research Institute for Interdisciplinary Sciences, School of Information Management and Engineering, Shanghai University of Finance and Economics.
	}
	\and Anthony Man-Cho So\thanks{Department of Systems Engineering and Engineering Management, The Chinese University of Hong Kong.
	}
}

\usepackage{amsopn}
\ifpdf
\hypersetup{
	pdftitle={},
	pdfauthor={}
}
\fi

\begin{document}
	
	\maketitle
	
	% REQUIRED
	\begin{abstract}
The heteroscedastic probabilistic principal component analysis (PCA) technique, a variant of the classic PCA that considers data heterogeneity, is receiving more and more attention in the data science and signal processing communities. In this paper, to estimate the underlying low-dimensional linear subspace (simply called \emph{ground truth}) from available heterogeneous data samples, we consider the associated non-convex maximum-likelihood estimation problem, which involves maximizing a sum of heterogeneous quadratic forms over an orthogonality constraint (HQPOC). We propose a first-order method---generalized power method (GPM)---to tackle the problem and establish its \emph{estimation performance} guarantee. Specifically, we show that, given a suitable initialization, the distances between the iterates generated by GPM and the ground truth decrease at least geometrically to some threshold associated with the residual part of certain ``population-residual decomposition''. In establishing the estimation performance result, we prove a novel local error bound property of another closely related optimization problem, namely quadratic optimization with orthogonality constraint (QPOC), which is new and can be of independent interest. Numerical experiments are conducted to demonstrate the superior performance of GPM in both Gaussian noise and sub-Gaussian noise settings.
	\end{abstract}
	
	% REQUIRED
	%\begin{keywords}
	%example
	%\end{keywords}
	
\section{Introduction}
\label{s:intro}
Principal component analysis (PCA) is a classic yet powerful method for dimensionality reduction. It can be derived from a celebrated probabilistic model on the observed data called the probabilistic PCA \cite{tipping1999probabilistic,bishop2006pattern}. One typical assumption in probabilistic PCA is that the noise level of all data samples remains the same. However, due to varying conditions of data acquisition, heterogeneous quality among data samples (i.e., with different noise variances) is ubiquitous in practice. For instance, in the field of air quality monitoring (see \cite{hong2021heppcat} and the references therein), two kinds of sensors for collecting data might exist. Some sensors are of high quality and are regularly maintained. Thus, the data obtained are very accurate. Other sensors are of medium or low quality. As a result, the obtained data can be noisier. One naive approach to coping with data heterogeneity is to pretend it is not there and apply PCA. Despite the sound solvability of PCA, its performance can be substantially degraded when it is applied to heterogeneous data samples because PCA treats all samples uniformly \cite{hong2018asymptotic,hong2023optimally}. To tackle such kind of heterogeneity, a new probabilistic model, called the heteroscedastic probabilistic PCA (HPPCA), has been proposed and studied recently; see, e.g., \cite{hong2019probabilistic,hong2021heppcat,gilman2022semidefinite}. Although HPPCA has been shown to yield performance gains in applications with heterogeneous noise levels across data samples, the resulting non-convex optimization problem gives rise to new computational challenges. Specifically, if we assume that the heterogeneous noise variances and the underlying signal strength are known and focus on estimating the underlying ground-truth low-dimensional linear subspace (see \cite[Section 4.2.2]{gilman2022semidefinite}), then the associated maximum-likelihood estimation (MLE) problem involves maximizing a sum of heterogeneous quadratic forms with an orthogonality constraint (HQPOC), which is quite challenging and cannot be directly solved by performing the singular value decomposition (SVD). We note in passing that while our paper focuses on heterogeneous noise across data samples, there are various papers considering heterogeneous noise across features; see, e.g., \cite{zhang2022heteroskedastic,leeb2021matrix,leeb2021optimal,zhou2023deflated}.

\paragraph{\textbf{Related works}} There are mainly two classes of numerical approaches for tackling HPPCA or general HQPOC problems: \emph{First-order methods} and \emph{semidefinite relaxation-based methods}. On one hand, since HPPCA generalizes the classic PCA, \cite{bolla1998extrema,breloy2021majorization} proposed a variant of the well-known power-type methods, called the generalized power method (GPM) (cf. \cite{journee2010generalized}), to tackle the non-convex HPPCA/HQPOC problem and showed that GPM would output certain critical points. The advantage of GPM is that it runs fast and can be applied to large-scale problems, as each iteration of GPM involves only matrix-vector multiplication and projection onto the Stiefel manifold. Despite GPM's simplicity and excellent empirical performance in various applications, its theoretical underpinning is still limited. Existing results on GPM for HPPCA or general HQPOC problems do not guarantee that the iterates converge to a global optimum, let alone the (linear) convergence rate of the iterates. Similar theoretical limitations exist when it comes to other first-order methods, such as the Riemannian gradient descent method (RGD) \cite{absil2009optimization,hu2020brief,boumal2023introduction}. On the other hand, following the elegant methodology of semidefinite relaxation (SDR) for general non-convex quadratically constrained quadratic programs \cite{luo2010semidefinite,nesterov2000semidefinite}, the authors in \cite{gilman2022semidefinite} proposed a novel SDR for HQPOC and showed that it is tight under certain non-trivial conditions (meaning that a global optimal solution to the non-convex HQPOC can be found by solving the convex SDR). Moreover, they verified the tightness conditions for the HPPCA problem. However, solving large-scale semidefinite programs (SDPs) is rather computationally costly. To alleviate the computational burden, the authors in \cite{gilman2022semidefinite} developed an alternative two-step strategy, where an iterative method is applied to obtain a candidate stationary point and then the global optimality of such a point is certified by solving a smaller SDP feasibility problem. Nevertheless, the computational cost is still high compared with pure first-order methods such as GPM and RGD. The above discussions motivate us to pursue the best of both approaches and ask the following question: 
%\begin{tcolorbox}[colframe={red},colback={white}]
% \vspace{-0.2cm}
% \begin{tcolorbox}[colframe={white},colback={white}] 
	Can we establish \emph{strong theoretical guarantees} for \emph{lightweight} first-order methods (e.g., GPM) when applied to the HPPCA problem?
	% \end{tcolorbox}
% \vspace{-0.3cm}
\paragraph{\textbf{Our contributions}} 
In this paper, we establish the \emph{estimation performance} guarantee of GPM (\Cref{thm: est_performace}) when it is applied to the HPPCA problem. Specifically, given a carefully designed initialization (e.g., by PCA), the \emph{estimation error} of the iterates of GPM, which is defined as the distances between the iterates generated by GPM and the ground truth, decreases at least \emph{geometrically} to some threshold. This provides, for the first time, an answer to the aforementioned question. To that end, we start by identifying a useful ``\emph{population-residual decomposition}'' (\Cref{lemma: snd}) of the HPPCA optimization problem.
Here, the population part corresponds to the HPPCA problem with \emph{infinite} data sample observations (see, e.g., \cite{chen2019gradient}) and turns out to be the well-studied quadratic optimization problem with orthogonality constraint (QPOC) \cite{liu2019quadratic,so2011moment}. Focusing on the QPOC, we show that it possesses a \emph{local error bound} property (\Cref{thm: gpmeb}), i.e., the distance of every point (near a global maximizer) to the set of global maximizers of QPOC can be bounded by certain residual measure associated with the GPM (for QPOC). This result is significant, as our local error bound property characterizes the growth of the objective function around the set of global maximizers and implies, among other things, that when applying the GPM to QPOC with a good initialization, the iterates generated by GPM converge \emph{linearly} to a global maximizer of QPOC. We remark that the authors in \cite{liu2019quadratic} showed a local error bound result for QPOC but with the residual measure being the Frobenius norm of the Riemannian gradients. This result was then employed to analyze the convergence rate of RGD. By contrast, we aim to analyze the convergence rate of GPM so that a different local error bound property with the optimality residual measure being related to GPM is required. Besides, our newly established local error bound result complements existing local error bound results in \cite{liu2017estimation,liu2019quadratic,wang2021linear,zhu2021orthogonal,chen2021local} for non-convex manifold optimization problems and can be of independent interest. Then, by incorporating the residual part into the population part and considering the original HPPCA problem, we show that with an initialization obtained by the classic PCA, the distances between the iterates generated by the GPM (for HPPCA) and the underlying ground-truth low-dimensional linear subspace decrease at least geometrically to some threshold associated with the residual part. Numerical experiments are conducted to validate our theoretical findings and demonstrate the superior performance of GPM.

% Our work also contributes to the emerging \emph{provable} non-convex optimization area via simple first-order methods. In particular, despite the non-convex and heterogeneous nature of the HPPCA problem, we analyze the \emph{estimation performance} of GPM for the HPPCA problem and show that the GPM solves the HPPCA problem efficiently given a qualified initialization. Moreover, prior to our work, the analyses of the generalized power method are performed on non-convex optimization problems with \emph{homogeneous} nature; see, e.g., \cite{boumal2016nonconvex,zhong2018near,chen2018projected,ling2022improved,liu2023unified,zhu2021orthogonal,wang2021optimal,araya2022seeded}. In sharp contrast,
% our analysis is performed on an optimization problem of \emph{heterogeneous} nature, which expands the repertoire of provable non-convex optimization problems.

\section{Preliminaries and Population-Residual Decomposition}
\label{s:ppca}
In this section, we first introduce the probabilistic model of HPPCA (in \eqref{eq: hppca_pm}) and provide the associated non-convex MLE formulation (in \eqref{eq: hppca-mle}). Then, to tackle such a non-convex MLE problem, we establish a specific
\emph{population-residual decomposition} for HPPCA in \Cref{lemma: snd}, which will play a crucial rule in the theoretical analysis part.

% Then, based on the obtained signal-noise decomposition, we characterize certain useful properties (in particular, the error bound property in Fact \ref{lemma:EB_grad} and the \emph{quadratic growth} property in Lemma \ref{lemma: QG}) of the optimization problem associated with the signal part and provide an upper bound on the distance between every global optimal solution to the MLE problem and the ground truth (see Lemma \ref{lemma: dist_gt}).

\subsection{Heteroscedastic probabilisitic PCA}
The HPPCA model \cite{hong2019probabilistic,hong2021heppcat,gilman2022semidefinite} assumes that $L$ known data groups of $n_1,\dots,n_L$ samples with given noise variances $v_1,\dots,v_L$, respectively are generated from the following probabilistic model:
\begin{equation} \label{eq: hppca_pm}
	\vvy_{l,i} = \bm{Q} \bm{\Theta} \vvz_{l,i} + \bm{\eta}_{l,i} \in \Br^d, \quad \forall l \in [L], i\in [n_l].
\end{equation}
Here, $[L] = \{1,2,\dots,L\}$; $[n_l] = \{1,2,\dots,n_l \}$; $\bm{Q}\in {\rm St}(d,K)=\{ \bX \in \mathbb{R}^{d \times K}| \bX^\top \bX = \bm{I}_K, d>K \}$ represents the underlying subspace (called the \emph{ground truth}); $\bm{\Theta} = {\rm diag}([\sqrt{\lambda_1}, \dots, \sqrt{\lambda_K}]^\top) \in \mathbb{R}^{K \times K} $ with $\lambda_1 > \lambda_2 > \dots > \lambda_K > 0$ being the known signal strength; $\vvz_{l,i} \overset{\rm iid}{\sim} \NCal(\bm{0},\bm{I}_K)$ are latent variables; and $\bm{\eta}_{l,i}\overset{ \rm iid}{\sim} \NCal(\bm{0},v_l \bm{I}_d)$ with $v_1 > v_2 > \dots > v_L >0$ are additive heteroscedastic Gaussian noises that are independent of $\{ \vvz_{l,i}| l\in [L], i \in [n_l]\}$. We use $n=\sum_{l=1}^L n_l$ to denote the total number of data samples.

Under the above setting, we know $\vvy_{l,i} \overset{\rm iid}{\sim} \mathcal{N}(\bm{0}, \bm{Q} \bm{\Theta}^2 \bm{Q}^\top + v_{l} \bm{I}_{d})$. Through similar derivations to that of \cite[Section \uppercase\expandafter{\romannumeral2}]{hong2021heppcat}, the maximum-likelihood estimate of the underlying subspace $\bm{Q}$ (see \cite[Section 4.2.2]{gilman2022semidefinite}) is a solution to the following problem of maximizing a sum of heterogeneous quadratic functions with an orthogonality constraint (HQPOC) \cite{bolla1998extrema,wang2022maximizing}:
\begin{align} 
	\label{eq: hppca-mle}
	\max_{\bm{X} \in {\rm St}(d,K)} \left\{ \sum_{l=1}^L \tr\left( \bm{X}^\top \frac{1}{n v_l} \bm{Y}_l \bm{Y}_l^\top \bm{X} \diag\left([w_{l,1},\dots,w_{l,K}]^\top \right)  \right)  =    \sum_{k=1}^K {\bx_k}^\top \bA_k \bx_k \right\}.
\end{align}
Here, $\bX = [\bx_1,\dots,\bx_K] \in {\rm St}(d,K) \subseteq \Br^{d\times K}$, $w_{l,k} = \frac{\lambda_k}{\lambda_k + v_l} \in (0,1)$, $\bm{Y}_l = [\vvy_{l,1},\dots,\vvy_{l,n_l}]$, $\diag\left([w_{l,1},\dots,w_{l,K}]^\top \right)$ is a diagonal matrix with the entries of $[w_{l,1},\dots,w_{l,K}]^\top$ on its diagonal, and $\bA_k = \sum_{l=1}^L w_{l,k} \frac{1}{nv_l} \bm{Y}_{l} \bm{Y}_l^\top  \succeq \bm{0}$ (i.e., positive semidefiniteness). Contrary to the fact that PCA can be efficiently solved via SVD, the non-convex optimization problem \eqref{eq: hppca-mle} is challenging and cannot be solved by performing SVD due to the sum of heterogeneous quadratic objectives.

\subsection{Population-residual decomposition}
To analyze the non-convex optimization problem \eqref{eq: hppca-mle} with finite data samples from the underlying probabilistic model \eqref{eq: hppca_pm}, one typical approach (see, e.g., \cite[Section 2]{chen2019gradient}) is first to consider the case of infinite data samples (but with $L$ being a fixed constant). Specifically, we first consider another perhaps simpler optimization problem obtained by taking expectations over the observed data, which corresponds to the ``population part'' of the HPPCA problem. Then, viewing the discrepancy between the finite-sample and infinite-sample cases as the residual part readily yields the desired population-residual decomposition. Such a population-residual decomposition provides important insights for tackling the HPPCA problem.

% {\red We can also tackle the case of growing SNR, but the structure would be with two lines. To mention the joint diagonal.}

\begin{lemma}[\textbf{Population-residual decomposition}]
	Denote $\bM_k = \frac{1}{n}\sum_{l=1}^{L} \sum_{i=1}^{n_l}  \frac{w_{l,k}}{v_l}{\vvy}_{l,i}\left({\vvy}_{l,i} \right)^\top - \gamma_k \bI_d$ with $\gamma_k = \sum_{l=1}^L w_{l,k}\frac{n_l}{n}$ (note $\gamma_1 > \gamma_2 > \dots>\gamma_K)$ and $\Delta_k = \bM_k - a_k \bm{Q} \bm{\Theta}^2 \bm{Q}^\top $. Then, under the HPPCA model \eqref{eq: hppca_pm}, the optimization problem \eqref{eq: hppca-mle} is equivalent to
	\begin{align}
		\max_{\bm{X} \in {\rm St}(d,K)}  f (\bm{X}) & = \sum_{k=1}^K \vvx_k^\top \bM_k \vvx_k \label{eq: proheppca}  \\
		& =  \underbrace{{\rm tr} \left(\bX^\top \bm{Q} \bm{\Theta}^2 \bm{Q}^\top \bX {\rm diag}(\ba) \right)}_{g(\bX): \textbf{ \scriptsize Population part}} + \underbrace{\sum_{k=1}^K \vvx_k^\top \Delta_k \vvx_k}_{h(\bX): \textbf{ \scriptsize Residual part}}.\label{eq: proheppcasnr}
	\end{align}
	Here, $\bQ \in {\rm St}(d,K)$ is the ground truth, $ \ba = (a_1,a_2,\dots,a_K)^\top$ with $a_k = \sum_{l=1}^L w_{l,k}\frac{n_l}{n} \frac{1}{v_l}$ (note $a_1 > a_2> \dots >a_K>0$), and ${\rm diag}(\ba)$ is a diagonal matrix with the entries of $\ba$ on its diagonal.
	\label{lemma: snd}
\end{lemma}
\begin{proof}\,
	Let $\tilde{\vvy}^k_{l,i} = \sqrt{\frac{w_{l,k}}{v_l}} \vvy_{l,i}$. From the fact ${\vvy}_{l,i} \overset{ \rm iid}{\sim} \NCal \left(\bm{0}, \bm{Q} \bm{\Theta}^2 \bm{Q}^\top + v_l \bI_d \right)$, we obtain
	\begin{align} \label{eq: tilde_y^i-distribution}
		\tilde{\vvy}^k_{l,i} \overset{ \rm iid}{\sim} \NCal \left(\bm{0}, w_{l,k}\left(\frac{1}{v_l}\bm{Q} \bm{\Theta}^2 \bm{Q}^\top + \bI_d \right) \right), \;
		\bA_k = \frac{1}{n}\sum_{l=1}^{L} \sum_{i=1}^{n_l} \tilde{\vvy}^k_{l,i}\left(\tilde{\vvy}^k_{l,i}\right)^\top.
	\end{align}
	Taking expectations over data samples yields
	\begin{align*}
		\EE[\bA_k] = \sum_{l=1}^L w_{l,k} \frac{n_l}{n} \left(\frac{1}{v_l} \bm{Q} \bm{\Theta}^2 \bm{Q}^\top + \bI_d \right) = a_k \bm{Q} \bm{\Theta}^2 \bm{Q}^\top + \gamma_k \bI_d, \; k \in [K],
	\end{align*}
	where $a_k = \sum_{l=1}^L w_{l,k}\frac{n_l}{n} \frac{1}{v_l}$ and $\gamma_k =  \sum_{l=1}^L w_{l,k}\frac{n_l}{n}$. Here, we assume that the proportion $\frac{n_l}{n}, \forall l \in [L]$ remains fixed when $n$ increases. Due to the fact that $w_{l,1} > w_{l,2} > \dots > w_{l,K}$ for each fixed $l \in [L]$, we know $a_1 > a_2> \dots >a_K$ and $\gamma_1 > \gamma_2 > \dots>\gamma_K$.
	Let $\bar{\bm{A}}_{k} = \EE[\bA_k]-\gamma_k \bI_d$, $\bM_k = \bA_k - \gamma_k \bI_d$, and
	\begin{align}\label{eq: noise-formulation}
		\Delta_k = \bA_k -  \EE[\bA_k] = \frac{1}{n}\sum_{l=1}^{L} \sum_{i=1}^{n_l} \tilde{\vvy}^k_{l,i}\left(\tilde{\vvy}^k_{l,i}\right)^\top - \left( a_k \bm{Q} \bm{\Theta}^2 \bm{Q}^\top + \gamma_k \bI_d \right).
	\end{align}
	Problem \eqref{eq: hppca-mle} can be equivalently written as
	\begin{align}
		\max_{\bm{X} \in {\rm St}(d,K)}  f (\bm{X}) & = \sum_{k=1}^K \vvx_k^\top \bM_k \vvx_k =\sum_{k=1}^K \vvx_k^\top \left(\bar{\bA}_{k} + \Delta_k \right) \vvx_k \label{eq: proheppca11}  \\
		& = \sum_{k=1}^K \vvx_k^\top \bm{Q} \bm{\Theta}^2 \bm{Q}^\top \vvx_k \cdot a_k + \sum_{k=1}^K \vvx_k^\top \Delta_k \vvx_k  \nonumber \\
		& =  \underbrace{{\rm tr} \left(\bX^\top \bm{Q} \bm{\Theta}^2 \bm{Q}^\top \bX {\rm diag}(\ba) \right)}_{g(\bX)} + \underbrace{\sum_{k=1}^K \vvx_k^\top \Delta_k \vvx_k}_{h(\bX)}. \label{eq: proheppcasnr11}
	\end{align}
	This gives the desired population-residual decomposition.
\end{proof}

Since the underlying subspace $\bm{Q}$ is unknown, the problem formulation \eqref{eq: proheppcasnr} will only be used in the theoretical analysis. We can view the HPPCA problem \eqref{eq: proheppcasnr} as a perturbed version of the population part $g(\bX)$ with perturbation levels depending on $\{\Delta_k \}_{k=1}^K$.
As shown in \Cref{lemma:ineq} and \Cref{lemma:ineq-1} below, the operator norms of the residual terms $\{ \Delta_k \}_{k=1}^K$ can be made arbitrarily small by increasing the number of data samples $n$.

\begin{lemma}(\cite[Lemma I.2]{gilman2022semidefinite})
	Let $\boldsymbol{y}_{1}, \ldots, \boldsymbol{y}_{n} \subseteq \mathbb{R}^{d}$ be i.i.d. centered Gaussian random variables with covariance operator $\boldsymbol{\Sigma}$ and sample covariance $\hat{\boldsymbol{\Sigma}}=\frac{1}{n} \sum_{i=1}^{n} \boldsymbol{y}_{i} \boldsymbol{y}_{i}^{\top}$. Then, there exists some constant $c_1>0$ such that with probability at least $1-e^{-t}$ for $t>0$,
	\begin{equation*}
		\|\hat{\bm{\Sigma}}- \bm{\Sigma}\| \leq c_1\| \bm{\Sigma}\| \max \left\{\sqrt{\frac{\tilde{r}(\bm{\Sigma}) \log d+t}{n}}, \frac{\left(\tilde{r}( \bm{\Sigma} ) \log d+t \right) \log n}{n}\right\},
	\end{equation*}
	where $\| \cdot \|$ represents the operator norm and $\tilde{r}(\bm{\Sigma}):=\tr(\bm{\Sigma}) /\|\bm{\Sigma}\|$.
	\label{lemma:ineq}
\end{lemma}
\begin{lemma}(\cite[Lemma C.4]{gilman2022semidefinite})
	Let $c_2>0$ be a universal constant and $\bm{C}_l = \frac{1}{n}\sum_{i=1}^{n_l} \bm{y}_{l,i} (\bm{y}_{l,i})^\top$. Then, with probability at least $1-e^{-t}$ for $t>0$,
	\begin{equation*}
		\left\|\bm{C}_l-\mathbb{E}\left[\bm{C}_l\right]\right\| \leq c_2 \left\|\mathbb{E}\left[\bm{C}_l\right]\right\| \max \left\{\sqrt{\frac{\frac{\bar{\xi}_{l}}{\bar{\sigma}_{l}} \log d+t}{n}}, \frac{\frac{\bar{\xi}_{l}}{\bar{\sigma}_{l}} \log d+t}{n} \log n\right\},
	\end{equation*}
	where $\bar{\sigma}_{l}=\left\|\mathbb{E}\left[\bm{C}_l\right]\right\|=\frac{n_{\ell}}{n}\left(\lambda_{1}+v_{\ell}\right) \;$ and $\; \bar{\xi}_{l}=\tr\left(\mathbb{E}\left[\bm{C}_l\right]\right)= \frac{n_{\ell}}{n}\left( \sum_{k=1}^{K} \lambda_{k}+v_{\ell}d\right).$
	\label{lemma:ineq-1}
\end{lemma} 
By \eqref{eq: noise-formulation}, $\{\Delta_k \}_{k=1}^K$ can be expressed as $\Delta_k = \sum_{l=1}^L \frac{w_{l,k}}{v_l} \left(\bm{C}_l-\mathbb{E}\left[\bm{C}_l\right]\right), \forall k \in [K]$. This, together with \Cref{lemma:ineq-1}, concludes that the operator norms of $\{\Delta_k \}_{k=1}^K$ can be arbitrarily small with growing $n$.

\section{Generalized Power Method}
We now present the lightweight generalized power method (GPM) \cite{journee2010generalized} for the HPPCA problem \eqref{eq: proheppca} in \Cref{alg:PGD_al}. An important step therein is to compute $\PCal_{\rm St}(\bX)$, which represents the projection of $\bm{X} \in \mathbb{R}^{d \times K}$ onto the non-convex manifold ${\rm St}(d,K)$, and it can be efficiently obtained by $\bU \bV^\top \in \PCal_{\rm St}(\bX)$ if $\bX$ admits the thin SVD $\bX = \bU {\bm \Sigma} \bV^\top$; see \cite{S66}. Due to the fact $\PCal_{\rm St}\left( \alpha \bm{X}^t +\left[\bM_1 \vvx_1^t, \dots, \bM_K \vvx_K^t \right]\right) = 	\PCal_{\rm St}\left(  \bm{X}^t  + \frac{1}{2\alpha} \nabla f(\bm{X}^t)  \right)$, GPM can be intuitively viewed as an instance of the non-convex \emph{projected gradient ascent} method with a constant stepsize $\frac{1}{2\alpha} >0 $. Besides, GPM closely resembles the classic power method for computing the dominant eigenvector of a matrix. In fact, if we let $\alpha=0$, then line 4 of \Cref{alg:PGD_al} consists of two steps: One is to compute the gradient and the other is to perform projection. These two steps are generalizations of the power iteration. It is worth mentioning that GPM has achieved significant success in the emerging \emph{provable} non-convex optimization area including group synchronization \cite{boumal2016nonconvex,zhong2018near,chen2018projected,ling2022improved,liu2023unified,zhu2021orthogonal}, community detection \cite{wang2021optimal,wang2023ptpm}, and graph matching \cite{onaran2017projected,araya2022seeded}.

% Given a qualified initialization $\bX^0$ (will be specified later), the GPM would iteratively refine the candidate point. We establish the \emph{estimation performance} as well as the linear rate of GPM for solving problem \eqref{eq: proheppca} in Theorem \ref{thm: est_performace}. 

\begin{algorithm}[!htbp]
	\caption{Generalized Power Method (GPM) for Solving Problem \eqref{eq: proheppca}}  
	\begin{algorithmic}[1]
		\State \textbf{Input:} Matrices $\bM_1,\dots,\bM_K$, the step size $\alpha > 0$
		\State \textbf{Initialize} an initial point $\bX^0 \in {\rm St}(d,K)$ satisfying certain conditions (see the condition (a) in \Cref{thm: est_performace})
		\For{$t=0,1,2,\dots,$}
		\State set $\bX^{t+1} \in \PCal_{\rm St}\left( \alpha \bm{X}^t +\left[\bM_1 \vvx_1^t, \dots, \bM_K \vvx_K^t \right]\right)$
		%		\IF{$\bH^{k+1} = \bH^k$}
		%		\STATE terminate and return $\bH^{k+1}$
		%		\ENDIF
		\EndFor
	\end{algorithmic}
	\label{alg:PGD_al}
\end{algorithm}

Due to the intrinsic non-convexity in the HPPCA problem, \Cref{alg:PGD_al} 
may not be effective for solving it unless a carefully designed initial point $\bX^0$ is
available. In the following, we describe two initialization schemes.
\paragraph{Initialization by PCA}
We can use the classic (homoscedastic) PCA to obtain a good initialization; see \cite[Section 3.2]{hong2019probabilistic}. Specifically, the starting point $\bX^0$ in \Cref{alg:PGD_al} is given by the $K$ principal eigenvectors of the sample covariance matrix $ \bm{C} = \frac{1}{n} \sum_{l=1}^L \sum_{i=1}^{n_l} \bm{y}_{l,i} (\bm{y}_{l,i})^\top$. The intuition behind this approach is that when taking expectation over the data samples, we obtain
\begin{align*}
	\mathbb{E}[\bm{C}] = \sum_{l=1}^L \frac{n_l}{n} \left(\bm{Q} \bm{\Theta}^2 \bm{Q}^\top + v_l \bI_d \right) = \bm{Q} \bm{\Theta}^2 \bm{Q}^\top + \sum_{l=1}^L \frac{n_l}{n} v_l \bI_d,
\end{align*}
\vspace{-0.2cm}
and $\bQ$ consists of $K$ principal eigenvectors of $\mathbb{E}[\bm{C}]$. 

\paragraph{Random initialization} Another way is to employ a random initialization directly. That is, we choose an initialization $\bX^0$ uniformly at random from the Stiefel manifold $\St(d,K)$. Empirically, such a random initialization works well.

% We first present the main result of this paper, which shows the estimation performance of Algorithm \ref{alg:PGD_al} for problem \eqref{eq: proheppca}. Notice that if $\hat{\bX}$ is a global optimal solution to \eqref{eq: proheppca}, $\hat{\bX} {\rm diag}(\bq)$ with $\bq \in \{1,-1\}^K$ is also a global optimal solution. To eliminate the effect of multiple optimal solutions, we define the distance between two orthonormal matrices ${\bX}, \bm{Q} \in \Br^{d\times K}$ as
% \begin{equation}\label{dist:XU}
	% 	d_F({\bX},\bm{Q}) = \min_{\bq \in \{1,-1\}^K} \| {\bX} - \bm{Q}{\rm diag}(\bq) \|_F.
	% \end{equation} 

\section{Theoretical Analysis and Main Results}
In this section, we study the estimation performance (\Cref{thm: est_performace}) of \Cref{alg:PGD_al} for the HPPCA problem. Specifically, we show that the iterates generated by \Cref{alg:PGD_al} approach the ground truth $\bQ$ at a geometric rate up to a certain threshold, which depends on the residual part in \eqref{eq: proheppcasnr}. Before we proceed, let us give an outline of our theoretical development.

Since the HPPCA problem \eqref{eq: proheppca} admits the population-residual decomposition \eqref{eq: proheppcasnr}, it can be seen as a slightly perturbed problem of the population part, which is a quadratic program with orthogonal constraints (QPOC),
\begin{equation}\label{P:min g} \tag{QPOC}
	\max_{\bm{X} \in {\rm St}(d,K)} g(\bX) = {\rm tr} \left(\bX^\top \bm{Q} \bm{\Theta}^2 \bm{Q}^\top\bX {\rm diag}(\ba) \right).
\end{equation}
Here, $\bm{Q},\bm{\Theta}$, and ${\rm diag}(\ba)$ are defined in \eqref{eq: hppca_pm} and \Cref{lemma: snd}, respectively. We first analyze problem \eqref{P:min g} and characterize its critical points in Lemma \ref{lem: 1stoptcon} as well as its (local) quadratic growth property in Lemma \ref{lemma: QG}. After that, we consider the following tailored GPM update with a stepsize $\alpha >0$ for \eqref{P:min g}:
\begin{align} \tag{GPM-QPOC}
	\bX^{t+1} & \in  \PCal_{\rm St}\left(\alpha \bm{X}^{t}+ \bm{Q} \bm{\Theta}^2 \bm{Q}^\top \bX^t\diag(\ba) \right) =  \PCal_{\rm St}\left(\mathcal{A}_{\alpha} (\bm{X}^{t})\right). \label{update:gpm-qpoc}
\end{align}
Here, $\mathcal{A}_{\alpha} (\bm{X}) = \alpha \bm{X}+ \bm{Q} \bm{\Theta}^2 \bm{Q}^\top \bX\diag(\ba)$, $\bX^0$ is given satisfying certain conditions, and $t = 0,1,\dots$ is the iteration number. We show in Theorem \ref{thm:LC for PGD} that the above iterative method enjoys a linear convergence rate. To prove Theorem \ref{thm:LC for PGD}, we derive a new and instrumental local error bound result for problem \eqref{P:min g} in Theorem \ref{thm: gpmeb}, which complements existing local error bound results in \cite{liu2019quadratic,wang2021linear,zhu2021orthogonal,chen2021local} and can be of independent interest. Lastly, the estimation performance guarantee of Algorithm \ref{alg:PGD_al} can be established by taking into account the discrepancy between problem \eqref{eq: proheppca} and problem \eqref{P:min g}.
%\begin{tcolorbox}
%\textbf{Step 1}: Basic properties of the \emph{signal part} (i.e., problem \eqref{P:min g}): First-order critical points; quadratic growth property; fixed points of \eqref{update:gpm-qpoc}.\\
%\textbf{Step 2}: Structural properties of the \emph{signal part}: Linear convergence result of \eqref{update:gpm-qpoc}; local error bound property. \\
%\textbf{Step 3}: Main results (see Theorem \ref{thm: est_performace}): Estimation performance and linear rate of GPM for the {HPPCA problem}.
%\end{tcolorbox}

\subsection{Basic properties of QPOC}

With the population-residual decomposition \eqref{eq: proheppcasnr} at hand, we provide several useful properties of problem \eqref{P:min g}. First, we characterize the set of first-order critical points of it using tools from smooth manifold optimization \cite{absil2009optimization,hu2020brief,boumal2023introduction}. We view the feasible set $\St(d,K)$ as an embedded submanifold of $\mathbb{R}^{d \times K}$ with the Euclidean inner product $\langle \cdot, \cdot \rangle$ as the Riemannian metric. The following results about critical points are similar to those in \cite[Section 4.8.2]{absil2009optimization} and \cite[Proposition 3]{liu2019quadratic}.

\begin{lemma}[\textbf{Set of critical points of QPOC}] \label{lem: 1stoptcon}
	$\bX \in \St(d,K)$ is a critical point of \eqref{P:min g} if and only if 
	\begin{align}
		\bm{Q} \bm{\Theta}^2 \bm{Q}^\top \bX\diag(\ba) = \bX \bm{S}
		\label{eq:first_order_symm}
	\end{align}
	for some symmetric matrix $\bm{S}$. Moreover, the columns of $\bX$ are the eigenvectors of $\bm{Q} \bm{\Theta}^2 \bm{Q}^\top$. That is, the set of critical points takes the form 
	\begin{equation}
		\bX = \bar{\bm{Q}} \bPi \diag(\bq),
		\label{eq:critical_eig}
	\end{equation}
	where $\bPi \in \br^{d \times K}$ has exactly one non-zero entry being $1$ in each column and further satisfies $\bPi^\top \bPi = \bI_K$, $\bq \in \{1,-1\}^K$, and $\bar{\bm{Q}} = [\bm{Q},\bm{Q}_{\perp}] \in \br^{d \times d}$ forms an orthogonal matrix.
\end{lemma}

\begin{proof}\, We first prove \eqref{eq:first_order_symm}.
	The Riemannian gradient of $g$ on the Stiefel manifold admits the expression
	\begin{equation*}
		\begin{aligned}
			\operatorname{grad} g(\bX) =2\left(\bI_d-\bX \bX^\top \right) \bm{Q} \bm{\Theta}^2 \bm{Q}^\top \bX  {\rm diag}(\ba)+ \bX \left(\bX^\top \bm{Q} \bm{\Theta}^2 \bm{Q}^\top \bX  {\rm diag}(\ba) -  {\rm diag}(\ba) \bX^\top \bm{Q} \bm{\Theta}^2 \bm{Q}^\top  \bX             \right).
		\end{aligned}
	\end{equation*}
	Since the columns of the first term in the expression of $\operatorname{grad} g(\bX)$ belong to the orthogonal complement of $\operatorname{span}(\bX)$ and the columns of the second term belong to $\operatorname{span}(\bX)$, it follows that $\operatorname{grad} g(\bX)$ vanishes if and only if
	\begin{equation}
		\left(\bI_d-\bX \bX^\top \right) \bm{Q} \bm{\Theta}^2 \bm{Q}^\top \bX  {\rm diag}(\ba)= \bm{0} 
		\label{eq:critical_lemma_1}
	\end{equation}
	and
	\begin{equation}
		\bX^\top \bm{Q} \bm{\Theta}^2 \bm{Q}^\top \bX {\rm diag}(\ba) -{\rm diag}(\ba) \bX^\top \bm{Q} \bm{\Theta}^2 \bm{Q}^\top  \bX  = \bm{0}.
		\label{eq:critical_lemma_2}
	\end{equation}
	Equation \eqref{eq:critical_lemma_1} yields
	$\left(\bI_d-\bX \bX^\top \right) \bm{Q} \bm{\Theta}^2 \bm{Q}^\top \bX = \bm{0}$,
	which further implies
	\begin{equation}
		\bm{Q} \bm{\Theta}^2 \bm{Q}^\top \bX=\bX \bm{M} 
		\label{eq:critical_lemma_3}
	\end{equation}
	for a symmetric matrix $\bm{M} = \bX^\top \bm{Q} \bm{\Theta}^2 \bm{Q}^\top \bX$. Then, \eqref{eq:first_order_symm} can be proved by setting $\bm{S} = \bm{M}{\rm diag}(\ba)$, which is symmetric due to \eqref{eq:critical_lemma_2}. Showing the converse direction (i.e., to show \eqref{eq:critical_lemma_1} and \eqref{eq:critical_lemma_2} using \eqref{eq:first_order_symm}) is straightforward.
	
	To prove \eqref{eq:critical_eig}, we observe that \eqref{eq:critical_lemma_2} implies $\bX^\top \bm{Q} \bm{\Theta}^2 \bm{Q}^\top \bX$ is diagonal. This further implies that $\bm{M}$ in \eqref{eq:critical_lemma_3} is diagonal. Hence, the columns of $\bX$ are eigenvectors of $\bm{Q} \bm{\Theta}^2 \bm{Q}^\top$ from \eqref{eq:critical_lemma_3}, which admit the expression in \eqref{eq:critical_eig}. Showing conversely that every such $\bX$ is a critical point of $g$ is straightforward (i.e., to show that \eqref{eq:critical_lemma_1} and \eqref{eq:critical_lemma_2} hold).
\end{proof}

Despite the existence of a collection of critical points (possibly infinite number of critical points due to $\bm{Q}_{\perp}$), the set of optimal solutions to \eqref{P:min g} admits a simple expression; see the result below.
\begin{lemma}[\textbf{Optimal solutions of QPOC}]
	The set of optimal solutions to \eqref{P:min g} is $\{\bm{Q} \diag(\bq) \}$ with $\bq \in \{1,-1\}^K$. Moreover, the distance between every optimal solution and any other critical points is no less than $\sqrt{2}$.
\end{lemma}
\begin{proof}\,
	According to \Cref{lem: 1stoptcon}, every global optimal point $\hat{\bX}$ takes the form $\hat{\bX} = \bar{\bm{Q}} \bPi \diag(\bq)$ for some $\bPi$ and $\bq$. First, we prove that every column vector of the matrix $\hat{\bX} \diag(\bq)^{-1}$ does not belong to $\bm{Q}_{\perp}$. Suppose that one column of $\hat{\bX}$ is a vector in the column space of $\bm{Q}_{\perp}$. To simplify notation, we consider $\hat{\bX}=\left[\bm{B},\bm{b}\right]$, where $\bm{B}$ contains the column vectors of $\bQ$ and $\bm{b}$ is the vector in the column space of $\bQ_{\perp}$. It follows that
	\begin{equation}
		\label{ineq:optimal_expression}
		\begin{aligned}
			g(\hat{\bX}) & ={\tr}\left(\left[\begin{array}{l}
				\bm{B}^{\top} \\
				\bm{b}^{\top}
			\end{array}\right] \bm{Q} \bm{\Theta}^2 \bm{Q}^\top \left[\bm{B},\bm{b}\right] {\diag}(\bm{a})\right)
			=\tr\left(\left(\begin{array}{c}
				\bm{B}^{\top} \bQ \\
				\bm{0}
			\end{array}\right) \bm{\Theta}^2\left[\bm{Q}^{\top} \bm{B}, \bm{0} \right] \diag(\bm{a})\right) \\
			& =\tr\left(\left[\begin{array}{cc}
				\bm{B}^{\top} \bm{Q} \bm{\Theta}^2 \bm{Q}^\top \bm{B} & \bm{0} \\
				\bm{0} & 0
			\end{array}\right] \diag(\bm{a})\right) 
			%		=\tr \left(\left[\begin{array}{llll}
				%				\lambda_1 & & & \\
				%				& \ddots & \lambda_k & \\
				%				& & & 0
				%			\end{array}\right] \diag(a)\right) \\
			%			&
			<\sum_{k=1}^K \lambda_k a_k = g(\bm{Q} \diag(\bq)).
		\end{aligned}    
	\end{equation}
	We conclude that if there exists some column vector of $\hat{\bX}$ belongs to $\bm{Q}_{\perp}$, the inequality will hold strictly and the associated function value will be strictly less than $g(\bm{Q} \diag(\bq))$.
	
	Second, we prove there is no such optimal point $\hat{\bX}$ that the columns of $\hat{\bX}$ are permutations of columns of $\bQ$. Suppose $\hat{\bX}=\bQ \cdot \bm{P}$, where $\bm{P} \in \mathbb{R}^{K \times K}$ is a permutation matrix and $\bm{P} \neq \bI_K$. It holds that
	\begin{equation}
		\begin{aligned}
			g(\hat{\bX}) =\tr \left(\bm{P}^{\top} \bQ^{\top} \bQ \bm{\Theta}^2 \bQ^{\top} \bQ \bm{P} \diag(\bm{a})\right) =\tr\left(\bm{P}^{\top} \bm{\Theta^2} \bm{P} \diag(\bm{a})\right)< \sum_{k=1}^K \lambda_k a_k
		\end{aligned} 
	\end{equation}
	The last inequality follows from the fact that $a_1>a_2>\dots>a_K >0$ and $\lambda_1>\lambda_2>\dots>\lambda_K >0$. Therefore, we obtain the desired result of the optimal solution set. Showing that the distance between every optimal solution and any other critical points is no less than $\sqrt{2}$ is straightforward.
\end{proof}

To eliminate the effect of multiple optimal solutions, we define the distance between two orthonormal matrices ${\bX}, \bm{Q} \in \Br^{d\times K}$ as 
\begin{equation}\label{dist:XU1}
	d_F({\bX},\bm{Q}) = \min_{\bq \in \{1,-1\}^K} \| {\bX} - \bm{Q}{\rm diag}(\bq) \|_F.
\end{equation} 
Observe that $d_F(\bX,\bm{Q}) = 0$ is equivalent to $\left\| |\bX^\top \bm{Q}| - \bI_K \right\|_F = 0$, where $|\cdot|$ is the component-wise absolute value function. In addition, $d_F^2({\bX},\bm{Q}) = 2 (K- {\rm tr}(|{\bX}^\top \bm{Q}| ))$.

Next, we provide an existing local error bound result \cite[Theorem 2]{liu2019quadratic}, which characterizes the growth behavior of the objective function in \eqref{P:min g} around each optimal solution.
\begin{fact}[\textbf{Local error bound}]
	Consider problem \eqref{P:min g}. There exist $\delta_3 \in (0,\frac{\sqrt{2}}{2})$ and $\eta >0$ such that for all $\bX \in \St(d,K)$ with $d_F(\bX,\bm{Q})\leq \delta_3$,
	\begin{align}\label{eq: globalER}
		d_F(\bX,\bm{Q})\leq \eta \cdot \| \grad g(\bX) \|_F,
	\end{align}
	where $\grad g(\bX) = \left(\bm{I}_d - \frac{1}{2}\bX \bX^\top \right) \left(\nabla g(\bx) - \bX \nabla g(\bX)^\top \bX \right)$.
	\label{lemma:EB_grad}
\end{fact}
Based on the above local error bound result, we derive a quadratic growth property of problem \eqref{P:min g}, which, to the best of our knowledge, is new and can be of independent interest. We remark that the quadratic growth property (\Cref{lemma: QG} and \Cref{cor: gqg}) will play an important role in proving Theorem \ref{thm: est_performace}. In fact, various non-convex statistical estimation problems including phase/group synchronization \cite{boumal2016nonconvex,zhu2021orthogonal} and dictionary learning \cite{shen2020complete} rely on the quadratic growth property in theoretical developments.
\begin{lemma}[\textbf{Local quadratic growth}] \label{lemma: QG} The following local quadratic growth property holds for problem \eqref{P:min g}:
	\begin{align}\label{eq: quad-grow}
		g(\bm{Q}) -g(\bX) \geq \frac{1}{4\eta} \cdot d_F^2(\bX,\bm{Q}), \quad \forall \bX \in \St(d,K), d_F(\bX,\bm{Q}) \leq \beta_1
	\end{align}
	with some $\beta_1 \in (0,\frac{2}{3}\delta_3)$.
\end{lemma}
\begin{proof}\,
	Since $\beta_1 < \frac{\sqrt{2}}{2}$, without loss of generality, we assume $d_F(\bX, \bm{Q}) = \| \bX - \bm{Q}\|_F$. Suppose that \eqref{eq: quad-grow} does not hold for every (small) $\beta_1>0$. Then, there exists an $\bX^0 \in \St(d,K)$ satisfying $\| \bX^0 - \bm{Q}\|_F \leq \beta_1$ with $\beta_1 = \frac{2}{3} \delta_3$ such that
	\begin{equation}
		g(\bm{Q}) \leq g(\bX^0) + \frac{1}{4\eta} \| \bX^0 - \bm{Q}\|_F^2 -t_0  \text{ and } t_0 >0.
	\end{equation}
	By the above inequality, we know that $\bX^0 \neq \bm{Q}$. Let $\lambda_0 = \frac{1}{2} \| \bX^0- \bm{Q}\|_F$ and $\tau_0 = \frac{1}{4\eta} \| \bX^0- \bm{Q}\|_F^2 - t_0>0$.
	From Ekeland's variational principle (see, e.g., \cite[Theorem 2.26]{mordukhovich2006variational}), there exists $\bm{Z}^0 \in \St(d,K)$ such that $\| \bm{Z}^0 - \bX^0 \|_F \leq \lambda_0$ and 
	$\bZ^0 = {\arg\max}_{\bX \in \St(d,K)} \left\{ g(\bX) - \frac{\tau_0}{\lambda_0} \| \bX- \bZ^0\|_F \right\}$.
	The optimality condition yields 
	\begin{equation}
		\bm{0} \in \nabla g(\bZ^0) + \mathcal{N}_{\St(d,K)}(\bZ^0) - \frac{\tau_0}{\lambda_0} \cdot \mathbb{B}(\bm{0};1),
	\end{equation}
	where $\mathcal{N}_{\St(d,K)}(\bZ^0)$ denotes the normal space of the Stiefel manifold at point $\bZ^0$ and $\mathbb{B}(\bm{0};1) \subset \mathbb{R}^{d \times K}$ represents the Euclidean ball with center $\bm{0}$ and radius $1$. From the optimality condition, we have $\| \grad g(\bZ^0) \|_F \leq \frac{\tau_0}{\lambda_0} $. We also have $\| \bZ^0 -\bm{Q}\|_F \leq \| \bZ^0 - \bX^0 \|_F + \| \bX^0 - \bm{Q}\|_F \leq \delta_3$. This, together with the error bound result \eqref{eq: globalER}, gives $\| \bZ^0- \bm{Q} \|_F \leq \eta\| \grad g(\bZ^0) \|_F$. In addition, noticing that
	\begin{equation}
		\| \bX^0 - \bm{Q} \|_F \leq \| \bX^0 - \bZ^0 \|_F + \| \bZ^0 - \bm{Q}\|_F  \leq \lambda_0 + \eta \| \grad g(\bZ^0) \|_F,
	\end{equation}
	which results in $\frac{1}{2} \| \bX^0 - \bm{Q} \|_F \leq \eta\frac{\tau_0}{\lambda_0}$.
	By the definition of $\lambda_0$ and $\tau_0$, this further yields
	\begin{equation}
		\frac{1}{4} \| \bX^0 - \bm{Q} \|^2_F \leq \eta \cdot (\frac{1}{4 \eta} \|\bX^0 -\bm{Q}\|_F^2 - t_0),
	\end{equation}
	which is a contradiction since $\eta > 0$ and $t_0 >0$.
\end{proof}

The obtained local quadratic growth property can be globalized without too much effort, which is similar to the proof idea of \cite[Corollary 1]{liu2019quadratic}.
\begin{corollary}[\textbf{Global quadratic growth}] \label{cor: gqg} There exists an $\bar{\eta}>0$ such that for all $\bX \in \St(d,K)$,
	\begin{align}
		g(\bm{Q}) - g(\bX) \geq \bar{\eta} \cdot d_F^2(\bX, \bm{Q}).
	\end{align}
\end{corollary}

Now, we turn our focus to the analysis of GPM-QPOC. Since the update \eqref{update:gpm-qpoc} can be seen as a fixed-point iteration for problem \eqref{P:min g}, we characterize its fixed points and show that fixed points are critical points.

% \begin{lemma}[Fixed points]
	%	A point $\bX \in \St(d,K)$ is a fixed point of 
	% \eqref{update:gpm-qpoc} if and only if  $\tr\left( \bX^\top % \mathcal{A}_{\alpha}(\bX) \right) = \| \mathcal{A}_{\alpha}% (\bX)\|_*$.
	%\end{lemma}
	% 第4个 proof
	% \begin{proof}
		%	If $\bX$ is a fixed point, we have $ \bX \in \PCal_{\rm St}%(\mathcal{A}_{\alpha}(\bX))$, or equivalently, $ %\mathcal{A}_{\alpha}(\bX) = \bX \bm{B}$
		%	with $\bm{B}$ being a positive semidefinite matrix. This % implies 
		%	\begin{equation}
			%		\tr\left( \bX^\top \mathcal{A}_{\alpha}(\bX) \right)  = %\tr(\bm{B}) =  \| \mathcal{A}_{\alpha} (\bX)\|_*.
			%	\end{equation}
		%	For the ``if'' part, by Von Neumann's trace inequality, we %know
		%	\begin{equation}
			%		\tr\left( \bX^\top \mathcal{A}_{\alpha}(\bX) \right) %\leq \sum_{k=1}^K \sigma_k(\bX)\cdot %\sigma_k(\mathcal{A}_{\alpha}(\bX)) = \| \mathcal{A}_{\alpha}%(\bX)\|_*.
			%	\end{equation}
		%	Furthermore, the inequality holds with equality if and only % if $\bX \in \PCal_{\rm St}(\mathcal{A}_{\alpha}(\bX))$.
		% \end{proof}
	
	\begin{lemma}[\textbf{Fixed points of GPM-QPOC}]
		\label{lem: fixedpoint}
		A point $\bX \in \St(d,K)$ is a fixed point of 
		\eqref{update:gpm-qpoc} if and only if  $\tr\left( \bX^\top  \mathcal{A}_{\alpha}(\bX) \right) = \| \mathcal{A}_{\alpha} (\bX)\|_*$, where $\|\cdot \|_*$ represents the nuclear norm. Furthermore, every fixed point of \eqref{update:gpm-qpoc} is a first-order critical point of problem \eqref{P:min g}.
	\end{lemma}
	\begin{proof}\, If $\bX \in \St(d,K)$ is a fixed point, we have $ \bX \in \PCal_{\rm St}(\mathcal{A}_{\alpha}(\bX))$, or equivalently, $ \mathcal{A}_{\alpha}(\bX) = \bX \bm{B}$
		with $\bm{B}$ being a positive semidefinite matrix \cite{absil2012projection}. This implies 
		\begin{equation}
			\tr\left( \bX^\top \mathcal{A}_{\alpha}(\bX) \right)  = \tr(\bm{B}) = \| \mathcal{A}_{\alpha} (\bX)\|_*.
		\end{equation}
		For the ``if'' part, by Von Neumann's trace inequality, we know
		\begin{equation}
			\tr\left( \bX^\top \mathcal{A}_{\alpha}(\bX) \right) \leq \sum_{k=1}^K \sigma_k(\bX)\cdot \sigma_k(\mathcal{A}_{\alpha}(\bX)) = \| \mathcal{A}_{\alpha}(\bX)\|_*.
		\end{equation}
		Furthermore, the inequality holds with equality if and only if $\bX \in \PCal_{\rm St}(\mathcal{A}_{\alpha}(\bX))$. 
		
		%	We next prove the equivalence between fixed points of \eqref{update:gpm-qpoc} and critical points of \eqref{P:min g}. 
		If $\bX$ is a fixed point of \eqref{update:gpm-qpoc}, then we know $ \bX \in \PCal_{\rm St}(\mathcal{A}_{\alpha}(\bX))$, or equivalently, $\mathcal{A}_{\alpha}(\bX) = \bX \bm{B}$ with $\bm{B}$ being a symmetric positive semidefinite matrix. This implies
		\begin{equation}
			\bm{Q} \bm{\Theta}^2 \bm{Q}^\top \bX\diag(\ba)=\mathcal{A}_{\alpha}(\bX)-\alpha \bX = \bX(\bm{B}-\alpha \bI_K).
		\end{equation}
		From the first-order optimality condition in Lemma \ref{lem: 1stoptcon}, $\bX$ is a first-order critical point of \eqref{P:min g}.
		%	For each critical point of \eqref{P:min g} (i.e., $ \bar{\bm{Q}}\bPi \diag(\bq)$), the following equalities hold: 
		%	\begin{align*}
			%		\bm{Q} \bm{\Theta}^2 \bm{Q}^\top \bar{\bm{Q}}\bPi \diag(\bq) \diag(\ba) & = \bar{\bm{Q}} \bPi \diag(\bq) \cdot \diag(\bq) \bPi^\top {\rm BlkDiag}(\bm{\Theta}^2,\bm{0}) \bPi \diag(\bq) \diag(\ba) \\
			%		&=\bar{\bm{Q}} \bPi \diag(\bq) \cdot \bPi^\top {\rm BlkDiag}(\bm{\Theta}^2,\bm{0}) \bPi \diag(\ba),
			%	\end{align*}
		%\jx{Why the first equality?}
		%	where ${\rm BlkDiag}(\bm{\Theta}^2,\bm{0})$ denotes the block diagonal matrix whose diagonal blocks are $\bm{\Theta}^2$ and $\bm{0}$. Observe that the matrix $\bPi^\top {\rm BlkDiag}(\bm{\Theta}^2,\bm{0}) \bPi \diag(\ba)$ is positive semidefinite (in fact, $\bPi^\top {\rm BlkDiag}(\bm{\Theta}^2,\bm{0}) \bPi$ is a diagonal matrix), and hence we conclude that every critical point is a fixed point of \eqref{update:gpm-qpoc}.
	\end{proof}
	\subsection{Error bound and linear convergence rate of GPM-QPOC}
	In this subsection, we prove that given a suitable initialization, the sequence of iterates and the associated sequence of objective values generated by \eqref{update:gpm-qpoc} for problem \eqref{P:min g} will converge linearly to a global maximizer and the optimal value, respectively. Specifically, we have the following result.
	
	\begin{theorem}[\textbf{Linear convergence of GPM-QPOC}] \label{thm:LC for PGD}
		Consider problem \eqref{P:min g}. Given a suitable initial point $\bX^0$ (see Proposition \ref{prop:local_region}(a) for details), if we apply \eqref{update:gpm-qpoc} with $0 <\alpha < \lambda_K a_K - \lambda_1 a_1 \delta$ (see Theorem \ref{thm: gpmeb} for the definition of $\delta>0$) to solve it, then the iterates converge linearly to a global maximizer. That is,
		\begin{equation} \label{eq: qpoc-gpm-linear}
			g(\bQ) - g(\bX^{t+1}) \leq \left(g(\bQ) - g(\bX^t)\right) \cdot \gamma \text{ and }\; d_F(\bX^t, \bQ) \leq a \cdot \left(g(\bQ) - g(\bX^0)\right)^{1/2}\cdot \left(\sqrt{\gamma}\right)^{t},
		\end{equation}
		where $a > 0, \gamma \in (0,1)$ are fixed constants.
	\end{theorem}
	
	The proof of Theorem \ref{thm:LC for PGD} consists of three main parts. The first, which is the most challenging part, is to establish the following local error bound property for problem \eqref{P:min g}. Roughly speaking, such an error bound provides a computable estimate of the distance from every point to the set of global maximizers, which could be of independent interest.

	Notice that if $\bX \in \St(d,K)$ is a fixed point of the update \eqref{update:gpm-qpoc} and $\mathcal{A}_{\alpha}(\bX)$ admits the SVD
	\begin{align*}
		\mathcal{A}_{\alpha}(\bX) = \bm{U}_{\mathcal{A}_{\alpha}(\bX)} \bm{\Sigma}_{\mathcal{A}_{\alpha}(\bX)} \bV_{\mathcal{A}_{\alpha}(\bX)}^\top,
	\end{align*}
	then $\bX= \bm{U}_{\mathcal{A}_{\alpha}(\bX)} \bV_{\mathcal{A}_{\alpha}(\bX)}^\top$, or equivalently, $\bX \bV_{\mathcal{A}_{\alpha}(\bX)} \bm{\Sigma}_{\mathcal{A}_{\alpha}(\bX)} \bV_{\mathcal{A}_{\alpha}(\bX)}^\top - \mathcal{A}_{\alpha}(\bX) = \bm{0}$. Such an observation naturally motivates us to define the (optimality) residual function $\rho_\alpha(\bX)$ associated with fixed points of \eqref{update:gpm-qpoc} as 
	\begin{equation}
		\rho_\alpha(\bX) = \| \bX \bV_{\mathcal{A}_{\alpha}(\bX)} \bm{\Sigma}_{\mathcal{A}_{\alpha}(\bX)} \bV_{\mathcal{A}_{\alpha}(\bX)}^\top - \mathcal{A}_{\alpha}(\bX)\|_F.
	\end{equation}
	We derive the following local error bound result in terms of the newly defined residual function $\rho_\alpha(\bX)$.
	
	\begin{theorem}[\textbf{Local error bound of GPM-QPOC for QPOC}] There exist $\eta_1 >0$ and a constant $\delta \in (0,\frac{\sqrt{2}}{2} )$ such that for all $\bX \in \St(d,K)$ with $d_F(\bX,\bQ)\leq \delta$,
		\begin{equation}\label{eq: qpoc-errorbound-pgd}
			d_F(\bX,\bQ) \leq \eta_1 \cdot \rho_\alpha(\bX),
		\end{equation}
		where the stepsize $\alpha$ arising in $\rho_\alpha(\bX)$ is required to satisfy $0 \le \alpha < \lambda_K a_K - \lambda_1 a_1 \delta$ in \eqref{update:gpm-qpoc}.
		\label{thm: gpmeb}
	\end{theorem}
	\begin{proof}\,
		We start by providing a lower bound of $\rho_\alpha(\bX)$,
		\begin{align}
			\rho_\alpha(\bX) & = \| \bX \bV_{\mathcal{A}_{\alpha}(\bX)} \bm{\Sigma}_{\mathcal{A}_{\alpha}(\bX)} \bV_{\mathcal{A}_{\alpha}(\bX)}^\top - \mathcal{A}_{\alpha}(\bX)\|_F \nonumber \\
			& = \| \bX \bV_{\mathcal{A}_{\alpha}(\bX)}
			\bm{\Sigma}_{\mathcal{A}_{\alpha}(\bX)} \bV_{\mathcal{A}_{\alpha}(\bX)}^\top - \bU_{\mathcal{A}_{\alpha}(\bX)} \bm{\Sigma}_{\mathcal{A}_{\alpha}(\bX)} \bV_{\mathcal{A}_{\alpha}(\bX)}^\top\|_F \nonumber \\
			& = \| \bX \bV_{\mathcal{A}_{\alpha}(\bX)} \bm{\Sigma}_{\mathcal{A}_{\alpha}(\bX)}  - \bU_{\mathcal{A}_{\alpha}(\bX)} \bm{\Sigma}_{\mathcal{A}_{\alpha}(\bX)} \|_F \nonumber\\
			& \geq  \sigma_K\left( \mathcal{A}_{\alpha}(\bX)\right) \| \bX \bV_{\mathcal{A}_{\alpha}(\bX)} - \bU_{\mathcal{A}_{\alpha}(\bX)}\|_F,  \label{eq: rhogeq}
		\end{align}
		where $\bU_{\mathcal{A}_{\alpha}(\bX)} \in \mathbb{R}^{d \times K} $ is the left singular vectors of $\mathcal{A}_{\alpha}(\bX)$ and $\sigma_K\left( \mathcal{A}_{\alpha}(\bX)\right)$ is the $K$-th largest singular value of $\mathcal{A}_{\alpha}(\bX)$. We claim $\sigma_K\left( \mathcal{A}_{\alpha}(\bX)\right) > 0 $. This can be shown as follows.
		
		Denote the $K$-th largest singular value of $\bQ \bm{\Theta}^2 \bQ^\top \bX \diag(\ba)$ by $\bar{\sigma}_K$ and assume $d_F(\bX,\bQ) = \|\bX-\bQ \|_F$. We obtain
		\begin{align*}
			\bar{\sigma}_K & \geq \sigma_K( \bQ \bm{\Theta}^2 \diag(\ba)) - \| \bQ \bm{\Theta}^2 \bQ^\top (\bX - \bQ) \diag(\ba) \|\\
			& \geq \lambda_K a_K - \lambda_1 a_1 \| \bX - \bQ\| \geq \lambda_K a_K - \lambda_1 a_1 \| \bX - \bQ\|_F \\
			& \geq \lambda_K a_K - \lambda_1 a_1 \delta >0,
		\end{align*}
		where the first inequality comes from Weyl's inequality and the last is due to our assumption on the stepsize $\alpha$. It follows that 
		\begin{align} \label{eq: sigmaK >0}
			\sigma_K \left( \mathcal{A}_{\alpha}(\bX)\right)  \geq \bar{\sigma}_K - \| \alpha \bX \| \geq   \lambda_K a_K - \lambda_1 a_1 \delta - \alpha >0.
		\end{align}
		Next, according to \cite[Proposition 2]{liu2019quadratic}, we know 
		\begin{align*}
			\| \grad g(\bX) \| _F \leq \| \nabla g(\bX) - \bX \nabla g(\bX)^\top \bX \|_F.
		\end{align*}
		In addition, observe that
		\begin{align*}
			\frac{1}{2}\| \grad g(\bX) \|_F
			& \leq \frac{1}{2} \| \nabla g(\bX) - \bX \nabla g(\bX)^\top \bX \|_F \\
			& = \| \frac{1}{2}\nabla g(\bX) + \alpha \bX  - \bX (\frac{1}{2}\nabla g(\bX)^\top + \alpha \bX ^ \top)\bX\|_F \\
			& = \| \bU_{\mathcal{A}_{\alpha}(\bX)} \bm{\Sigma}_{\mathcal{A}_{\alpha}(\bX)} \bV_{\mathcal{A}_{\alpha}(\bX)}^\top - \bX \bV_{\mathcal{A}_{\alpha}(\bX)} \bm{\Sigma}_{\mathcal{A}_{\alpha}(\bX)} \bU_{\mathcal{A}_{\alpha}(\bX)}^\top \bX\|_F \\
			& \le \|  (\bX - \bU_{\mathcal{A}_{\alpha}(\bX)} \bV_{\mathcal{A}_{\alpha}(\bX)}^\top) \bV_{\mathcal{A}_{\alpha}(\bX)} \bm{\Sigma}_{\mathcal{A}_{\alpha}(\bX)} \bU_{\mathcal{A}_{\alpha}(\bX)}^\top \bX \|_F \\ & \quad + \| \bU_{\mathcal{A}_{\alpha}(\bX)} \bm{\Sigma}_{\mathcal{A}_{\alpha}(\bX)} \bU_{\mathcal{A}_{\alpha}(\bX)}^\top (\bX - \bU_{\mathcal{A}_{\alpha}(\bX)} \bV_{\mathcal{A}_{\alpha}(\bX)}^\top )                    \|_F \\
			& \le \|\bX- \bU_{\mathcal{A}_{\alpha}(\bX)}\bV_{\mathcal{A}_{\alpha}(\bX)}^\top   \|_F \cdot \left(\| \bV_{\mathcal{A}_{\alpha}(\bX)} \bm{\Sigma}_{\mathcal{A}_{\alpha}(\bX)} \bU^\top_{\mathcal{A}_{\alpha}(\bX)} \bX\| + \|\bU_{\mathcal{A}_{\alpha}(\bX)} \bm{\Sigma}_{\mathcal{A}_{\alpha}(\bX)} \bU^\top_{\mathcal{A}_{\alpha}(\bX)} \|\right) \\
			& \leq 2 \| \bm{\Sigma}_{\mathcal{A}_{\alpha}(\bX)} \| \cdot \|\bX- \bU_{\mathcal{A}_{\alpha}(\bX)}\bV_{\mathcal{A}_{\alpha}(\bX)}^\top \|_F,
		\end{align*}
		where the second equality is due to $\mathcal{A}_{\alpha}(\bX) = \frac{1}{2}\nabla g(\bX) + \alpha \bX$. This, together with Fact \ref{lemma:EB_grad} and \eqref{eq: rhogeq}, yields that for all $\bX \in \St(d,K)$ with $d_F(\bX,\bQ)\leq \delta$ (note $\delta \leq \delta_3$ and $\lambda_K a_K - \lambda_1 a_1 \delta - \alpha >0$), 
		\begin{align*}
			d_F(\bX,\bQ) & \leq \eta \cdot \| \grad g(\bX) \|_F \leq 4 \eta \cdot \| \bm{\Sigma}_{\mathcal{A}_{\alpha}(\bX)} \| \cdot \|\bX- \bU_{\mathcal{A}_{\alpha}(\bX)}\bV_{\mathcal{A}_{\alpha}(\bX)}^\top \|_F \\
			& \leq \frac{4 \eta}{\sigma_K\left( \mathcal{A}_{\alpha}(\bX)\right)} \cdot \| \bm{\Sigma}_{\mathcal{A}_{\alpha}(\bX)} \| \cdot \rho_\alpha(\bX).
		\end{align*}
		By letting $\eta_1 = \max_{ \bX \in \St(d,K), d_F(\bX,\bQ)\leq \delta} \left\{ \frac{4 \eta}{\sigma_K\left( \mathcal{A}_{\alpha}(\bX)\right)} \cdot \| \bm{\Sigma}_{\mathcal{A}_{\alpha}(\bX)} \| \right\} $, we readily obtain the desired result.
	\end{proof}
	We remark that error bounds have long played a critical role in analyzing the convergence rate of iterative algorithms; see, e.g., \cite{luo1993error,zhou2017unified,li2018calculus} and the references therein. Our established error bound result expands the repertoire of existing error bound results in \cite{liu2017estimation,liu2019quadratic,wang2021linear,zhu2021orthogonal,chen2021local} for non-convex (manifold) optimization problems. It is well worth pointing out that the existing error bound result in Fact \ref{lemma:EB_grad} uses the Frobenius norm of the Riemannian gradient (i.e., $\|\grad g(\bX)\|_F$) as a residual measure, which is then used to analyze the convergence rate of the Riemannian gradient descent method \cite{liu2019quadratic}. On the contrary, we aim to analyze the convergence rate of \eqref{update:gpm-qpoc}, so that the new residual measure $\rho_\alpha(\bX)$ associated with fixed points is required.
	
	Now, let us proceed to the second part of the proof of Theorem \ref{thm:LC for PGD}, which elucidates the following key properties of \eqref{update:gpm-qpoc} for problem \eqref{P:min g}.
	\begin{proposition}\label{prop:convergence_prep}
		Consider problem \eqref{P:min g}. The sequence of iterates $\{\bX^t\}_{t\geq 0}$ generated by \eqref{update:gpm-qpoc} with $\alpha >0$ satisfies:\\
		{\rm (a)} (\textbf{Sufficient ascent}) $g(\bX^{t+1}) - g(\bX^{t}) \geq \alpha \cdot \| \bX^t - \bX^{t+1}\|_F^2$,\\
		{\rm (b)} (\textbf{Cost-to-go estimate}) $g(\bQ) - g(\bX^t) \leq \beta_3 \cdot d_F^2(\bX^t,\bQ)$,\\
		{\rm (c)} (\textbf{Safeguard}) $\rho_\alpha(\bX^t) \leq \beta_4 \cdot\| \bX^{t+1} - \bX^t \|_F$, \\
		where $\beta_3>0$ and $\beta_4>0$ are some constants.
	\end{proposition}
	
	\begin{proof}\,
		We first prove the sufficient ascent property. Observe that
		\begin{align*}
			& \quad g(\bX^{t+1}) - g(\bX^{t}) 
			\\
			& = \tr\left( (\bX^{t+1})^{\top}\bm{Q} \bm{\Theta}^2 \bm{Q}^\top \bX^{t+1}\diag(\ba) \right)-\tr\left( (\bX^{t})^{\top}\bm{Q} \bm{\Theta}^2 \bm{Q}^\top \bX^{t}\diag(\ba) \right) 
			\\
			&=  \tr\left( (\bX^{t+1}-\bX^{t})^{\top}\bm{Q} \bm{\Theta}^2 \bm{Q}^\top (\bX^{t+1}-\bX^{t})\diag(\ba) \right) + \tr\left((\bX^{t+1}-\bX^{t})^{\top} \alpha (\bX^{t+1}-\bX^{t}) \right) \\
			& \quad + 2 \cdot \tr \left( (\bX^{t+1})^{\top}\bm{Q} \bm{\Theta}^2 \bm{Q}^\top \bX^{t}\diag(\ba) \right)-2 \cdot \tr \left( (\bX^{t})^{\top}\bm{Q} \bm{\Theta}^2 \bm{Q}^\top 	\bX^{t}\diag(\ba) \right) \\
			&\quad + 2 \cdot \tr \left((\bX^{t+1})^{\top} \alpha \bX ^{t} \right) - 2 \cdot \tr \left((\bX^{t})^{\top} \alpha \bX ^{t} \right).
		\end{align*}
		By the update of \eqref{update:gpm-qpoc}, it holds that
		\begin{align*}
			& \quad 2 \cdot \tr \left( (\bX^{t+1})^{\top}\bm{Q} \bm{\Theta}^2 \bm{Q}^\top \bX^{t}\diag(\ba) \right)-2 \cdot \tr \left( (\bX^{t})^{\top}\bm{Q} \bm{\Theta}^2 \bm{Q}^\top \bX^{t}\diag(\ba) \right)\\
			& \ge 2 \cdot \tr \left( (\bX^{t})^{\top} \alpha \bX ^{t} \right) - 2 \cdot \tr \left( (\bX^{t+1})^{\top} \alpha \bX ^{t} \right).
		\end{align*}
		Hence, we have
		\begin{align*}
			g(\bX^{t+1}) - g(\bX^{t}) & \geq  \tr\left( (\bX^{t+1}-\bX^{t})^{\top}\bm{Q} \bm{\Theta}^2 \bm{Q}^\top (\bX^{t+1}-\bX^{t})\diag(\ba) \right)\\
			& \quad + \tr\left((\bX^{t+1}-\bX^{t})^{\top} \alpha (\bX^{t+1}-\bX^{t}) \right) \\
			& \geq \alpha \| \bX^t - \bX^{t+1}\|_F^2.
		\end{align*}
		By \cite[Proposition 9]{liu2019quadratic}, the cost-to-go estimate holds.
		We next consider the property of safeguard,
		\begin{align*}
			\rho_\alpha(\bX^t) &= \| \bX^t \bV_{\mathcal{A}_{\alpha}(\bX^t)} \bm{\Sigma}_{\mathcal{A}_{\alpha}(\bX^t)} \bV_{\mathcal{A}_{\alpha}(\bX^t)}^\top - \mathcal{A}_{\alpha}(\bX^t)\|_F \\
			&= \| \bX^t \bV_{\mathcal{A}_{\alpha}(\bX^t)} \bm{\Sigma}_{\mathcal{A}_{\alpha}(\bX^t)} \bV_{\mathcal{A}_{\alpha}(\bX^t)}^\top - \bX^{t+1} \bV_{\mathcal{A}_{\alpha}(\bX^t)} \bm{\Sigma}_{\mathcal{A}_{\alpha}(\bX^t)} \bV_{\mathcal{A}_{\alpha}(\bX^t)}^\top  \|_F \\
			&= \| \left(\bX^t - \bX^{t+1} \right)\bV_{\mathcal{A}_{\alpha}(\bX^t)} \bm{\Sigma}_{\mathcal{A}_{\alpha}(\bX^t)} \bV_{\mathcal{A}_{\alpha}(\bX^t)}^\top   \|_F \\
			&\leq \| \bV_{\mathcal{A}_{\alpha}(\bX^t)} \bm{\Sigma}_{\mathcal{A}_{\alpha}(\bX^t)} \bV_{\mathcal{A}_{\alpha}(\bX^t)}^\top  \| \cdot \| 
			\bX^t - \bX^{t+1}\|_F.
		\end{align*}
		The proof is complete.
	\end{proof}
	
	We remark that in Proposition \ref{prop:convergence_prep}, the stepsize $\alpha >0$ is required to ensure the sufficient ascent property, which is slightly different from the requirement $\alpha \ge 0$ in Theorem \ref{thm: gpmeb}. Note that the local error bound result in Theorem \ref{thm: gpmeb} only holds in a local region. To be able to analyze the whole sequence of iterates generated by \eqref{update:gpm-qpoc} for problem \eqref{P:min g} and subsequently the whole sequence of iterates generated by Algorithm \ref{alg:PGD_al}, we need to show that the iterates always stay in such a local region. The following result constitutes the third part of the proof of Theorem \ref{thm:LC for PGD}.
	\begin{proposition}[\textbf{The iterates stay in a local region}]\label{prop:local_region}
		Suppose that \\
		{\rm (a)} (\textbf{Initialization}) $g(\bQ) - g(\bX^0) \leq \delta_1$, where $\delta_1 = \delta^{2} \bar{\eta}$ and $\delta,\bar{\eta}$ are defined in Theorem \ref{thm: gpmeb} and Corollary \ref{cor: gqg}, respectively; \\
		{\rm (b)} (\textbf{Residual setting}) $\sum_{k=1}^{K} \| \Delta_k\| \leq \frac{2 c (1-\gamma) \delta_1 }{2 L_g + (1-\gamma) \delta_1 } (< 2c)$, where $c = \lambda_K a_K - \lambda_1 a_1 \delta - \alpha > 0$ and $L_g $ is the Lipschitz constant of $g$ over the compact Stiefel manifold.\\
		Then, the sequence $\{\bX^t\}_{t \geq 0}$  generated by Algorithm \ref{alg:PGD_al} with $\alpha >0$ satisfies $ d_F(\bX^{t},\bQ) \leq \delta $.
	\end{proposition}
	\begin{proof}\,
		The result will be shown by induction. When $t=0$, we know by Corollary \eqref{cor: gqg} and the assumption $g(\bQ) - g(\bX^0) \leq \delta_1$  that
		\begin{align}
			d_F(\bX^0,\bQ) \leq \sqrt{\left( g(\bQ) - g(\bX^0) \right)/\bar{\eta}} \leq \sqrt{\delta_1/\bar{\eta} } = \delta.
			\label{eq:induction_ball}
		\end{align}
		Next, we assume $g(\bQ) - g(\bX^ {t}) \leq \delta_1 $ for some $t \geq 0$, which implies $d_F(\bX^t,\bQ) \leq \delta$. According to the iterative procedures in \eqref{update:gpm-qpoc} and Algorithm \ref{alg:PGD_al}, we denote
		\begin{equation*}
			\bar{\bX}^{t+1} \in  \PCal_{\St}\left(\alpha \bX^{t}+ \bm{Q} \bm{\Theta}^2 \bm{Q}^\top \bX^t\diag(\ba) \right), \bX^{t+1} \in\PCal_{\rm St}\left( \alpha \bm{X}^t +\left[\bM_1 \vvx_1^t, \dots, \bM_K \vvx_K^t \right]\right).
		\end{equation*}
		By Proposition \ref{prop:convergence_prep}, Theorem \ref{thm: gpmeb}, and the fact $d_F(\bX^t,\bQ) \leq \delta$, it follows that
		\begin{align}
			g(\bQ) - g(\bX^t) & \leq \beta_3 \cdot d_F^2(\bX^t,\bQ)  \leq \eta_1 ^{2} \beta_3 \cdot \left(\rho_\alpha(\bX^t) \right)^{2} \nonumber \\
			& \leq \eta_1 ^{2} \beta_4 ^{2} \beta_3 \cdot \| \bar{\bX}^{t+1} - \bX^t \|_F ^{2} \leq \frac{\eta_1 ^{2} \beta_4 ^{2} \beta_3}{\alpha} \cdot \left(g(\bar{\bX}^{t+1}) - g(\bX^{t})\right).
			\label{eq: convergence_ineq}
		\end{align}
		Hence, we have
		\begin{align}
			g(\bQ) - g(\bar{\bX}^ {t+1}) & = g(\bQ) - g(\bX^t) - \left(g(\bar{\bX}^{t+1}) - g(\bX^{t}) \right) \nonumber\\
			& \leq \left(\frac{\eta_1 ^{2} \beta_4 ^{2} \beta_3}{\alpha}-1 \right) \left(g(\bar{\bX}^{t+1}) - g(\bQ) + g(\bQ) - g(\bX^{t})\right).
			\label{eq:convergence_1}
		\end{align}
		Since $g(\bQ)\geq g(\bX^{t})$  for all $t \geq 0$, without loss of generality we assume $a'=\frac{\eta_1 ^{2} \beta_4 ^{2} \beta_3}{\alpha}>1$. Then, from \eqref{eq:convergence_1}, one has
		\begin{align}
			g(\bQ) - g(\bar{\bX}^ {t+1}) \leq \frac{a'-1}{a'} \left(g(\bQ) - g(\bX^ {t})  \right),
			\label{eq: linear_convergence_0}
		\end{align}
		which implies
		\begin{align*}
			g(\bQ) - g( \bar{\bX}^ {t+1}) \leq  \left(g(\bQ) - g(\bX^ {t})  \right) \cdot \gamma \leq \gamma \cdot \delta_1  \text{ with }  \gamma=\frac{a'-1}{a'} \in (0,1).
		\end{align*}
		Note that
		\begin{align*}
			|g(\bar{\bX}^{t+1}) - g(\bX^{t+1})| &   \leq L_g \| \bar{\bX}^{t+1} - \bX^{t+1} \|_F \\
			& \leq 2 L_g \frac{\| \left[\Delta_1 \vvx_1^t, \dots, \Delta_K \vvx_K^t \right] \|_F}{ 2 (\lambda_K a_K - \lambda_1 a_1 \delta - \alpha) - \| \left[\Delta_1 \vvx_1^t, \dots, \Delta_K \vvx_K^t \right] \| },
		\end{align*}
		where $L_g >0$ is the Lipschitz constant of $g$ over the Stiefel manifold and the second inequality is due to the perturbation result in \cite[Theorem 2.4]{li2002perturbation} and \eqref{eq: sigmaK >0}. \\
		By the assumption {\rm (b)}, we have
		\begin{align*}
			|g(\bar{\bX}^{t+1}) - g(\bX^{t+1})|  &  \leq  2 L_g \frac{\| \left[\Delta_1 \vvx_1^t, \dots, \Delta_K \vvx_K^t \right] \|_F}{ 2 (\lambda_K a_K - \lambda_1 a_1 \delta - \alpha) - \| \left[\Delta_1 \vvx_1^t, \dots, \Delta_K \vvx_K^t \right] \| } \\
			& \leq   2 L_g \frac{\sum_{k=1}^{K} \| \Delta_k\|}{ 2 (\lambda_K a_K - \lambda_1 a_1 \delta - \alpha) - \sum_{k=1}^{K} \| \Delta_k\| }\\
			& \leq   \left( 1-\gamma \right) \cdot \delta_1,
		\end{align*}
		where the second inequality is due to 
		\begin{align*}
			\| \left[\Delta_1 \vvx_1^t, \dots, \Delta_K \vvx_K^t \right] \|_F^2 \leq \sum_{k=1}^K \| \Delta_k \vvx_k^t\|_2^2 \leq \sum_{k=1}^K \| \Delta_k \|^2 \leq (\sum_{k=1}^{K} \| \Delta_k\| )^2
		\end{align*}
		and 
		\begin{align*}
			\| \left[\Delta_1 \vvx_1^t, \dots, \Delta_K \vvx_K^t \right] \| \leq \| \left[\Delta_1 \vvx_1^t, \dots, \Delta_K \vvx_K^t \right] \|_F \leq  \sum_{k=1}^{K} \| \Delta_k\|.
		\end{align*}
		It follows that
		\begin{align*}
			g(\bQ) - g( \bX^ {t+1}) & \leq g(\bQ) - g( \bar{\bX}^ {t+1}) + |g(\bar{\bX}^{t+1}) - g(\bX^{t+1})| \\
			& \leq \gamma \cdot \delta_1 + \left(1-\gamma \right)\cdot \delta_1 = \delta_1.
		\end{align*}
		Then, through similar derivations in \eqref{eq:induction_ball}, we have
		\begin{align*}
			d_F(\bX^{t+1},\bQ) \leq \delta.
		\end{align*}
		This completes the proof.
	\end{proof}
	With \Cref{thm: gpmeb}, \Cref{prop:convergence_prep}, and \Cref{prop:local_region} at our disposal, we now present the proof of \Cref{thm:LC for PGD}.
	\begin{proof}\,
		Note that \eqref{eq: linear_convergence_0} implies
		$g(\bQ) - g(\bX^ {t}) \leq \left(g(\bQ) - g(\bX^ {0})\right) (\gamma)^t$ with $\gamma=\frac{a'-1}{a'} \in (0,1)$.
		Furthermore, from \eqref{eq: convergence_ineq} we know \begin{align*}
			d_F^2(\bX^t,\bQ) \leq  \frac{\eta_1 ^{2} \beta_4 ^{2}}{\alpha} \cdot \left(g(\bQ) - g(\bX^{t})\right)  \leq \frac{\eta_1 ^{2} \beta_4 ^{2}}{\alpha} \cdot \left(g(\bQ) - g(\bX^ {0})\right) \gamma^t,
		\end{align*}
		which gives
		\begin{align*}
			d_F(\bX^t,\bQ) \leq \frac{\eta_1  \beta_4 }{\alpha^{1/2}} \cdot \left(g(\bQ) - g(\bX^ {0})\right)^{1/2} (\sqrt{\gamma})^{t}.
		\end{align*}
		This completes the proof.
	\end{proof}	
	
	\subsection{Estimation performance of GPM in Algorithm \ref{alg:PGD_al}}
	Based on the above analysis of the \emph{population part} \eqref{P:min g}) and the update \eqref{update:gpm-qpoc}, we are ready to establish the estimation performance of GPM in Algorithm \ref{alg:PGD_al} for the HPPCA problem.
	\begin{theorem}[\textbf{Estimation performance}]
		Suppose that the following two conditions hold:\\
		{\rm (a)} (\textbf{Initialization}) $g(\bQ) - g(\bX^0) \leq \delta_1$, where $\delta_1 = \delta^{2} \bar{\eta}$, $\delta,\bar{\eta}$ are two constants defined in Theorem \ref{thm: gpmeb} and Corollary \ref{cor: gqg}, respectively, and the population part function $g$ is defined in \eqref{eq: proheppcasnr}; \\
		{\rm (b)} (\textbf{Residual setting}) $\sum_{k=1}^{K} \| \Delta_k\| \leq \frac{2 c (1-\gamma) \delta_1 }{2 L_g + (1-\gamma) \delta_1 } (< 2c)$, where $c = \lambda_K a_K - \lambda_1 a_1 \delta - \alpha > 0$, $\gamma \in (0,1)$ is a constant defined in Theorem \ref{thm:LC for PGD}, and $L_g $ is the Lipschitz constant of $g$ over $\St(d,K)$.\\
		Then, the iterates generated by Algorithm \ref{alg:PGD_al} with $\alpha>0$ approach the ground truth $\bQ$ at a geometric rate up to a certain threshold, namely,
		\begin{align} \label{eq: estimationperformance-lasttheorme}
			d_F^2(\bX^{t+1},\bQ) \leq \frac{1}{\bar{\eta}} \cdot \left(g(\bQ) - g(\bX^0)\right) \cdot \gamma^{t+1} + \frac{\delta_2}{(1-\gamma)\bar{\eta}},
		\end{align}
		where $\gamma \in (0,1)$ and $\delta_2 = \frac{ 2 L_g \cdot \sum_{k=1}^{K} \| \Delta_k\|}{ 2 (\lambda_K a_K - \lambda_1 a_1 \delta - \alpha) - \sum_{k=1}^{K} \| \Delta_k\| }$ is related to the residual part in \eqref{eq: proheppcasnr}.
		\label{thm: est_performace}
	\end{theorem}
	\begin{proof}\,
		Suppose that we have an iterate $\bX^t \in {\rm St}(d,K)$ and denote 
		\begin{equation*}
			\bar{\bX}^{t+1} \in  \PCal_{\St}\left(\alpha \bX^{t}+ \bm{Q} \bm{\Theta}^2 \bm{Q}^\top \bX^t\diag(\ba) \right), \bX^{t+1} \in\PCal_{\rm St}\left( \alpha \bm{X}^t +\left[\bM_1 \vvx_1^t, \dots, \bM_K \vvx_K^t \right]\right).
		\end{equation*}
		By Proposition \ref{prop:local_region}, it holds that $d_F(\bX^{t},\bQ) \leq \delta, \forall t \ge 0$.
		According to \eqref{eq: convergence_ineq}, \eqref{eq:convergence_1}, and \eqref{eq: linear_convergence_0}, we obtain 
		\begin{equation} \label{eq: ep-qpoc-linear}
			g(\bQ) - g(\bar{\bX}^{t+1}) \leq \left(g(\bQ) - g(\bX^t)\right) \cdot \gamma. 
		\end{equation}
		In addition, we estimate
		\begin{align*}
			|g(\bar{\bX}^{t+1}) - g(\bX^{t+1})| & \leq L_g \| \bar{\bX}^{t+1} - \bX^{t+1} \|_F \\
			& \leq 2 L_g \frac{\| \left[\Delta_1 \vvx_1^t, \dots, \Delta_K \vvx_K^t \right] \|_F}{ 2 (\lambda_K a_K - \lambda_1 a_1 \delta - \alpha) - \| \left[\Delta_1 \vvx_1^t, \dots, \Delta_K \vvx_K^t \right] \| } \\
			& \le 2 L_g \frac{\sum_{k=1}^{K} \| \Delta_k\|}{ 2 (\lambda_K a_K - \lambda_1 a_1 \delta - \alpha) - \sum_{k=1}^{K} \| \Delta_k\| },
		\end{align*}
		where $L_g >0$ is the Lipschitz parameter of $g$ over the compact Stiefel manifold and the second inequality is due to the perturbation result in \cite[Theorem 2.4]{li2002perturbation}.
		This, together with \eqref{eq: ep-qpoc-linear}, yields
		\begin{equation}
			g(\bQ) - g({\bX}^{t+1}) \leq \left(g(\bQ) - g(\bX^t)\right) \cdot \gamma + \delta_2,
		\end{equation}
		where $\delta_2$ is defined as $\delta_2 = \frac{ 2 L_g \cdot \sum_{k=1}^{K} \| \Delta_k\|}{ 2 (\lambda_K a_K - \lambda_1 a_1 \delta - \alpha) - \sum_{k=1}^{K} \| \Delta_k\| }$.
		By the global quadratic growth property in Corollary \ref{cor: gqg}, it follows that
		\begin{equation}
			d_F^2(\bX^{t+1},\bQ) \le \frac{1}{\bar{\eta}} \cdot  (g(\bQ) - g({\bX}^{t+1})) \leq \frac{1}{\bar{\eta}} \cdot \left(g(\bQ) - g(\bX^0)\right) \cdot \gamma^{t+1} + \delta_2 /\left((1-\gamma)\bar{\eta}\right).
		\end{equation}
		We readily obtain the desired result.
	\end{proof}

	Some comments on Theorem \ref{thm: est_performace} are in order.
	\begin{enumerate}
		\item[i)] Condition (a) requires that the initial point $\bX^0$ lies in a neighborhood of the ground truth $\bQ$. The size of the initialization region depends only on the population part \eqref{P:min g} and not on the residual part. Condition (b) imposes a bound on the residual part. It is worth noting that our results can handle problem \eqref{eq: proheppcasnr} with \emph{deterministic} $\{ \Delta_k \}_{k=1}^K$, which is much more general than the HPPCA problem with a probabilistic model.
		\item[ii)] From \eqref{eq: estimationperformance-lasttheorme}, we know that eventually, the iterates $\bX^t$ (and every accumulation point of $\{\bm{X}^t\}_{t\geq 0}$) would lie in a small neighborhood of the ground truth, the size of which depends on the residual part in \eqref{eq: proheppcasnr}.
		
		\item[iii)] There is no global optimality guarantee for Algorithm \ref{alg:PGD_al} in Theorem \ref{thm: est_performace}. Instead, since our goal in this paper is to provide a computable estimate of the unknown subspace $\bQ$ rather than to solve the challenging non-convex optimization problem \eqref{eq: proheppca} to global optimality, the estimation performance result in Theorem \ref{thm: est_performace} is good enough for such a purpose. Surely, establishing the global optimality guarantee for Algorithm \ref{alg:PGD_al} is itself an interesting theoretical problem (see \cite{zhu2021orthogonal,ling2022improved,ling2023near} for related works) and, as a byproduct, this may provide a tighter bound on $d_F(\bX^t,\bQ)$ (see the following Lemma \ref{lemma: dist_gt}). We leave this as a future research direction.
	\end{enumerate}
	\begin{lemma}[\textbf{Distance to the ground truth}]
		\label{lemma: dist_gt}
		Suppose that $\hat{\bX}$ is a global optimal solution to problem \eqref{eq: proheppcasnr}, we have
		\begin{equation}
			d_F(\hat{\bX},\bm{Q}) \leq \frac{2\sqrt{K}}{\bar{\eta} } \max_{k\in [K]} \|\Delta_k \|,
		\end{equation}
		where $\Bar{\eta} >0$  is the constant in Corollary \ref{cor: gqg}.
	\end{lemma}
	
	\begin{proof}
		\,From the global optimality of $\hat{\bX}$ of problem \eqref{eq: proheppcasnr}, we have
		\begin{equation*} 
			\sum_{k=1}^K \hat{\vvx}_k^\top (a_k \bm{Q} \bm{\Theta}^2 \bm{Q}^\top) \hat{\vvx}_k + \sum_{k=1}^K \hat{\vvx}_k^\top \Delta_k \hat{\vvx}_k
			\geq \sum_{k=1}^K \bm{Q}_k^\top (a_k \bm{Q} \bm{\Theta}^2 \bm{Q}^\top) \bm{Q}_k + \sum_{k=1}^K \bm{Q}_k^\top \Delta_k \bm{Q}_k.
		\end{equation*}
		Rearranging the terms gives
		\begin{align*}
			& \sum_{k=1}^K \bm{Q}_k^\top (a_k \bm{Q} \bm{\Theta}^2 \bm{Q}^\top) \bm{Q}_k - \sum_{k=1}^K \hat{\vvx}_k^\top (a_k \bm{Q} \bm{\Theta}^2 \bm{Q}^\top) \hat{\vvx}_k \\
			\leq &  \sum_{k=1}^K \hat{\vvx}_k^\top \Delta_k \hat{\vvx}_k - \sum_{k=1}^K \bm{Q}_k^\top \Delta_k \bm{Q}_k
			= \sum_{k=1}^K (\hat{\vvx}_k - \bm{Q}_k)^\top \Delta_k (\hat{\vvx}_k + \bm{Q}_k) \\
			\leq & \sum_{k=1}^K \| \hat{\vvx}_k - \bm{Q}_k\|_2 \cdot \| \Delta_k \| \cdot \| \hat{\vvx}_k + \bm{Q}_k\|_2
			\leq \left(\max_{k\in [K]} \|\Delta_k\| \right) \cdot \sqrt{K \sum_{k=1}^K(\| \hat{\vvx}_k - \bm{Q}_k\|_2^2 \cdot \| \hat{\vvx}_k + \bm{Q}_k\|_2^2)}\\
			\leq & 2\sqrt{K} \max_{k\in [K]} \|\Delta_k\| \cdot d_F(\hat{\bX},\bm{Q}).
		\end{align*}
		Combining the above with the global quadratic growth property in Corollary \ref{cor: gqg} yields
		\begin{align*}
			\bar{\eta} \cdot d_F^2(\hat{\bX},\bm{Q}) \leq  g(\bm{Q}) - g(\hat{\bX}) \leq 2\sqrt{K}\max_{k\in [K]} \|\Delta_k\|\cdot d_F(\hat{\bX},\bm{Q}),
		\end{align*}
		which implies
		\begin{align*}
			d_F(\hat{\bX},\bm{Q}) \leq \frac{2\sqrt{K}}{\bar{\eta}} \max_{k \in [K]} \|\Delta_k \|.
		\end{align*}
		This completes the proof.
	\end{proof}
	% \begin{remark}
		% One natural question arises: How good is the threshold in \eqref{eq: estimationperformance-lasttheorme}? To answer this, we provide an upper bound for the distance between each global optimal solution to \eqref{eq: proheppcasnr} and the ground truth $\bm{Q}$.
		% \end{remark}
	
	% \jx{To show that under some noise conditions, a small neighborhood of the ground truth has only one critical point and it is exactly the optimum. This would imply the convergence of GPM to the optimum to some extent.}

	Note that there is an initialization condition in Theorem \ref{thm: est_performace}. The following result shows that an initialization via the homogeneous PCA can satisfy this condition.
	\begin{lemma}[\textbf{Initialization}]
		Let the initial point $\bX^0$ be the $K$ principal eigenvectors of the sample covariance matrix $ \bm{C} = \frac{1}{n} \sum_{l=1}^L \sum_{i=1}^{n_l} \bm{y}_{l,i} (\bm{y}_{l,i})^\top$. If 
		\begin{equation} \label{eq: initlast}
			\Vert \bm{C}-\mathbb{E}[\bm{C}] \Vert^2 \leq \frac{\delta_1}{ \beta_3 \cdot 2^3 \sum \limits _{j=1} ^ {K} 1/\left(\min^2 \{\lambda_{j-1}-\lambda_{j},\lambda_{j}-\lambda_{j+1}\}\right) } 
		\end{equation}
		holds, then we have $g(\bQ) - g(\bX^0) \leq \delta_1$, where $\delta_1 = \delta^{2} \bar{\eta}$; $\delta,\bar{\eta}$ are two constants defined in Theorem \ref{thm: gpmeb} and Corollary \ref{cor: gqg}, respectively; $\lambda_0,\lambda_{K+1}$ are defined as $+\infty$ and $0$, respectively; and $\beta_3$ is the constant defined in Proposition \ref{prop:convergence_prep}.
		\label{thm:initialization}
	\end{lemma}

	\begin{proof}\,
		By Davis-Kahan theorem in \cite[Theorem 2]{yu2015useful}, we know
		\begin{equation*}
			\min \left \{\Vert \vvx^0_{j} - \bQ_{j} \Vert _{F}, \Vert \vvx^{0}_j + \bQ_{j}\Vert_{F} \right\}  \leq \frac{2^{\frac{3}{2}} \Vert \bm{C}-\mathbb{E}[\bm{C}] \Vert }{\min\{\lambda_{j-1}-\lambda_{j},\lambda_{j}-\lambda_{j+1}\}}, \quad \forall j\in [K].
		\end{equation*}
		Here, $\vvx_j^0$ denotes the $j$-th column of $\bX^0$, $\bQ_j$ denotes the $j$-th column of $\bQ$, and $\lambda_0,\lambda_{K+1}$ are defined as $+\infty$ and $0$, respectively. It follows that there exists some $ \diag{(\bm{q})} $ with $\bm{q} \in \{1,-1\}^{K} $ such that
		\begin{equation*}
			\Vert \bX^0  - \bQ \diag{(\bm{q})} \Vert _{F} ^{2} \leq  2^{3} \Vert \bm{C}-\mathbb{E}[\bm{C}] \Vert ^{2} \sum \limits _{j=1} ^ {K} \frac{1}{\min^2 \{\lambda_{j-1}-\lambda_{j},\lambda_{j}-\lambda_{j+1}\}}.
		\end{equation*}
		According to the result in \cite[Proposition 9]{liu2019quadratic} (or Proposition \ref{prop:convergence_prep}(b)), we have
		\begin{equation*} 
			\begin{aligned}
				g(\bQ)-g(\bX ^0) & = g(\bQ \diag(\bm{q}))-g(\bX ^0) \leq \beta_3 \cdot \Vert \bX ^0  - \bQ \diag(\bm{q}) \Vert _{F}^2 \\
				& \leq \beta_3 \cdot 2^3  \Vert \bm{C}-\mathbb{E}[\bm{C}] \Vert ^{2} \sum \limits _{j=1} ^ {K} \frac{1}{\min^2 \{\lambda_{j-1}-\lambda_{j},\lambda_{j}-\lambda_{j+1}\}}\\
				& \leq \delta_1.
			\end{aligned}
		\end{equation*}
		This concludes the proof. 
	\end{proof}
	
	The condition \eqref{eq: initlast} in Lemma \ref{thm:initialization} can be achieved when the sample size $n$ is large. This can be shown from Lemma \ref{lemma:ineq} and the following inequality 
	\[
	\Vert \bm{C}-\mathbb{E}[\bm{C}] \Vert \leq \sum \limits _{l=1} ^{L}  \Vert \bm{C}_l -\mathbb{E}[\bm{C}_l] \Vert  \quad \text{  with  } \quad  \bm{C}_l = \frac{1}{n}\sum_{i=1}^{n_l} \bm{y}_{l,i} (\bm{y}_{l,i})^\top.
	\]
	\section{Numerical Simulations}
	\label{s: experiment}
	In this section, we demonstrate the recovery performance of GPM in Algorithm \ref{alg:PGD_al} for the HPPCA problem under both Gaussian and general sub-Gaussian noises. In addition, we assess its robustness against noise and compare it with the classic PCA.
	
	\subsection{Estimation performance and linear rate}
	We consider $n = 1000$ samples with dimension $d=100$. There are $L=2$ data groups: The first $n_1 = 200$ samples are generated with a noise variance $v_1 = 1$ and the others ($n_2 = 800$) are generated with a noise variance $v_2 = 6$. We set the parameter $K$ to be $3$ and set the (diagonal) signal strength matrix $\bm{\Theta}^2$ to be $\diag([5,3.5,2]^\top)$.
	
	We test the performance of GPM in Algorithm \ref{alg:PGD_al} with both a random initialization and a carefully designed PCA initialization under the Gaussian noise setting. The stepsize parameter $\alpha$ is set to be $0.05$. \Cref{fig: EP}(a) shows the fast convergence behavior of Algorithm \ref{alg:PGD_al} with different initialization methods. Note that $\hat{\bX}$ represents the accumulation point of the iterates. Although we have not yet theoretically shown that $\hat{\bX}$ is a global maximizer of problem \eqref{eq: proheppca}, empirically, the obtained solution $\hat{\bX}$ by Algorithm \ref{alg:PGD_al} is indeed a global maximizer, which is certified by solving a feasibility problem according to \cite[Theorem 4.1]{gilman2022semidefinite}. As can be seen from \Cref{fig: EP}(a), the iterates generated by GPM would converge to a global maximizer of problem \eqref{eq: proheppca} with a linear rate. Besides, we show the estimation performance of Algorithm \ref{alg:PGD_al} in \Cref{fig: EP}(b), which demonstrates that the distances between the iterates $\{\bX^t\}_{t\ge 0}$ and the ground truth $\bQ$ decrease to some threshold. This phenomenon is consistent with our theoretical findings in Theorem \ref{thm: est_performace}. In addition, as can be seen from both \Cref{fig: EP}(a) and \Cref{fig: EP}(b), compared with the random initialization strategy, adopting the PCA initialization renders better progress because it is closer to the limiting point $\hat{\bX}$ and the ground truth $\bQ$. 
	
	% After conducting the experiments described above, a natural question arises: Can we effectively handle Sub-Gaussian noise?
	% From a theoretical standpoint, our analysis does not consider the Sub-Gaussian noise condition in \eqref{eq: hppca_pm} ,
	% which renders the MLE formulation in \eqref{eq: hppca-mle} invalid. 
	We also test the performance of GPM
	in Algorithm \ref{alg:PGD_al} with both random initialization and PCA initialization under the more practical sub-Gaussian noise setting, even though the formulation in \eqref{eq: hppca-mle} is no longer an MLE. Specifically, the noises of the first data group are sampled from the uniform distribution over the interval $[-\sqrt{3v_1}, \sqrt{3v_1}]$ with $v_1 = 1/2$. Similarly, the noises of the second data group are sampled from the uniform distribution over the interval $[-\sqrt{3v_2}, \sqrt{3v_2}]$ with $v_2 = 3$. Other settings in this experiment are identical to those in the above experiments with Gaussian noises. Similar to the results in the Gaussian case, \Cref{fig: EP}(c) shows that the iterates generated by GPM would converge linearly to a global maximizer of problem \eqref{eq: proheppca} and \Cref{fig: EP}(d) shows GPM can estimate the ground truth well even under non-Gaussian noises.
	
	\begin{figure}[ht]
		\centering
		\begin{minipage}[b]{0.24\linewidth}
			\centering
			\centerline{\includegraphics[width=\linewidth]{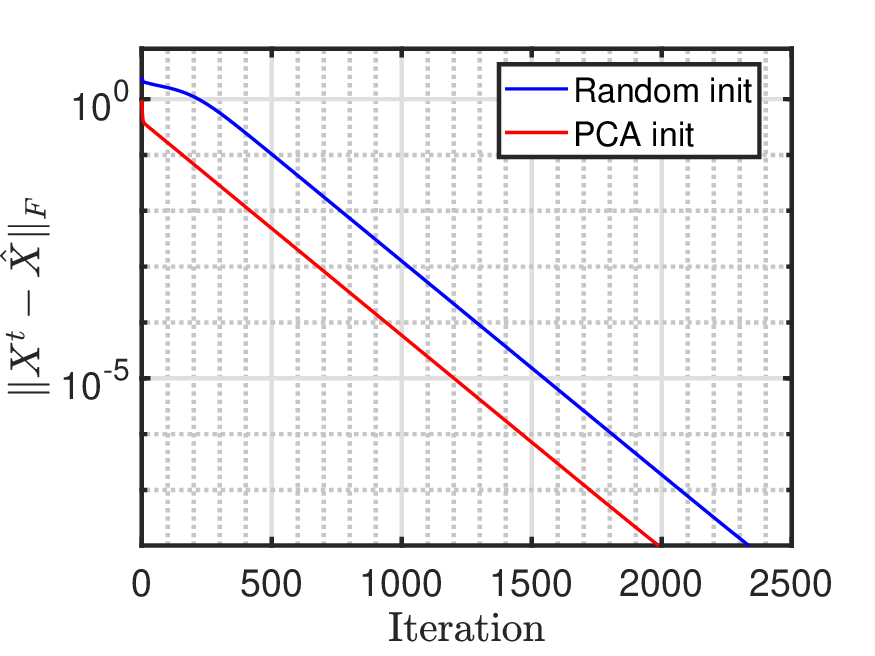}}
			\centerline{(a)}
		\end{minipage}
		\begin{minipage}[b]{0.24\linewidth}
			\centering
			\centerline{\includegraphics[width=\linewidth]{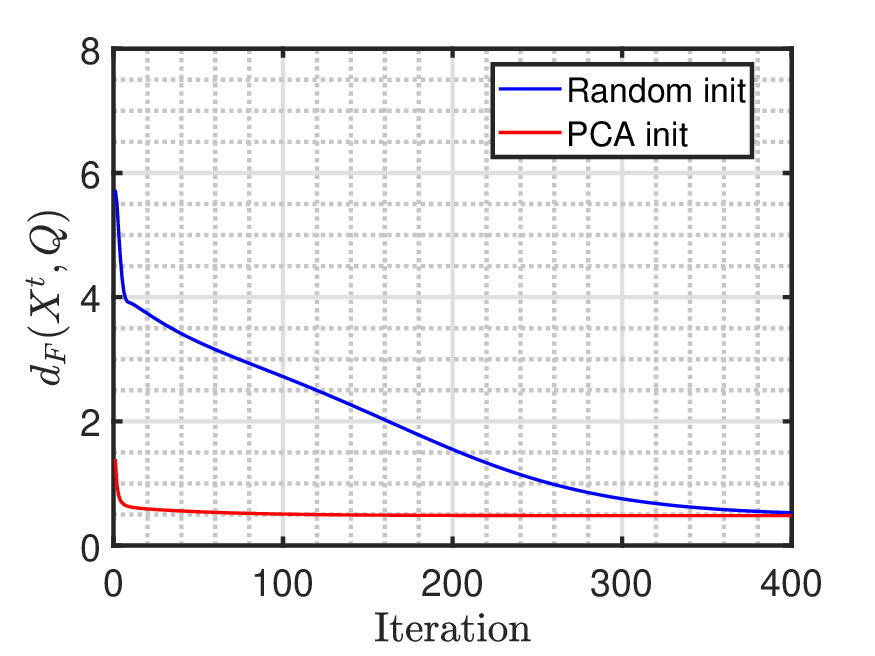}}
			\centerline{(b)}
		\end{minipage}
		\begin{minipage}[b]{0.24\linewidth}
			\centering
			\centerline{\includegraphics[width=\linewidth]{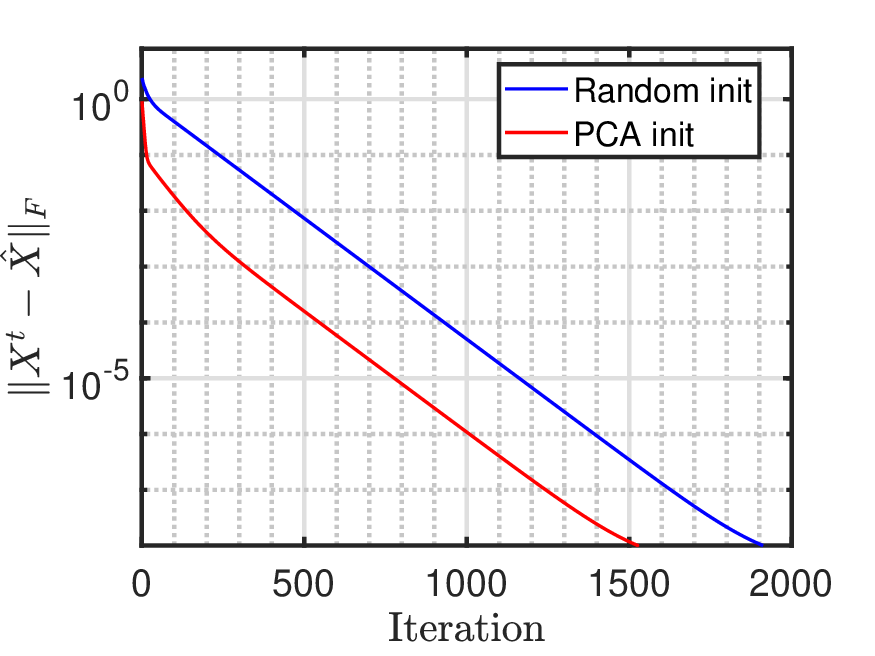}}
			\centerline{(c)}
		\end{minipage}
		\begin{minipage}[b]{0.24\linewidth}
			\centering
			\centerline{\includegraphics[width=\linewidth]{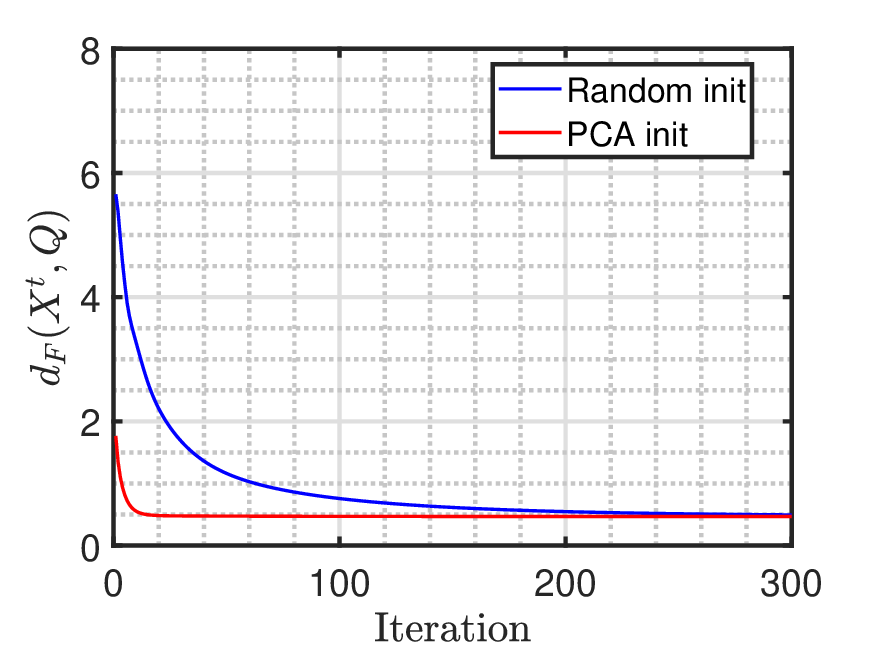}}
			\centerline{(d)}
		\end{minipage}
		\caption{Numerical performance of GPM for HPPCA. The left two figures are under Gaussian noises and they show the linear convergence to a global maximizer and estimation performance, respectively. The right two figures are under sub-Gaussian noises.}
		\label{fig: EP}
	\end{figure}
	
	\subsection{Assessing robustness}
	This experiment also considers $n = 1000$ samples with dimension $d=100$. We assess the robustness of HPPCA/GPM against increasing noise and increasing heterogeneity. The error measure is defined as the subspace distance between the ground truth $\bQ$ and the outputs of different models.
	
	First, if the sample noise level is low, then directly applying PCA may attain satisfactory results. Thus, we compare the performances of PCA and HPPCA with increasing noise levels in \Cref{fig: rob_noise}. We set different noise strength levels in the experiment. Concretely, for the $i$-th noise strength level, the noise variance vector is set to be $v^i = v+\frac{i}{10}\times v$ with $v=\left[\frac{1}{10},\frac{6}{10}\right]$. \Cref{fig: rob_noise} shows that HPPCA can output better estimators than PCA. Second, we test the robustness of HPPCA against increasing noise heterogeneity. For the $i$-th noise strength level, the noise variance vector is $v^i=\left[\frac{1}{10},\frac{6}{10}+\frac{i}{10}\right]$. \Cref{fig:rob_heteo} shows that increasing noise heterogeneity hardly influences the performance of HPPCA while the performance of PCA is quite sensitive to it.
	
	\begin{figure}[ht]
		\begin{minipage}[t]{0.48\linewidth} 
			\centering 
			\includegraphics[scale=0.34]{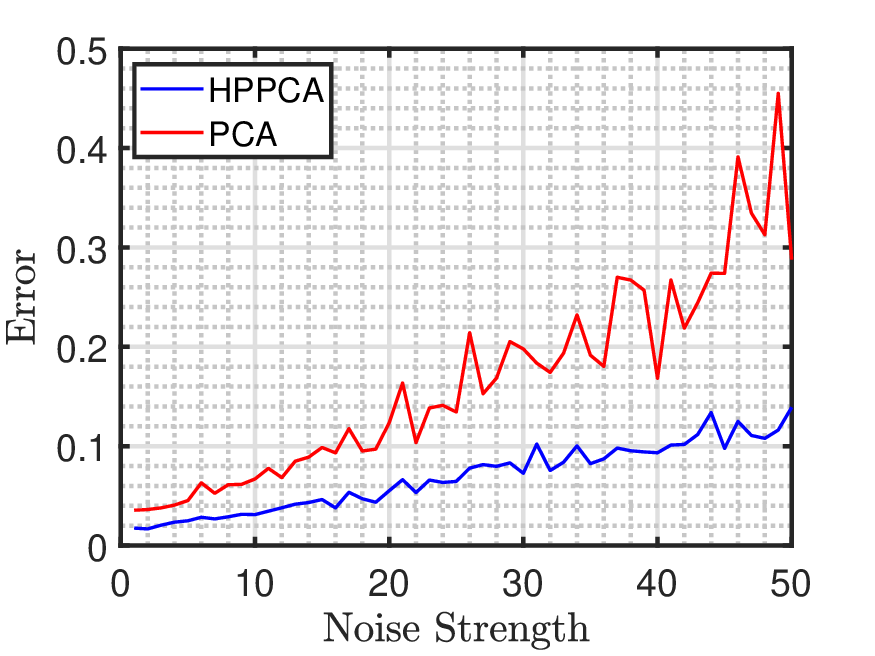} 
			\caption{Robustness against noise.} 
			\label{fig: rob_noise}
		\end{minipage}% 
		\begin{minipage}[t]{0.48\linewidth} 
			\centering 
			\includegraphics[scale=0.34]{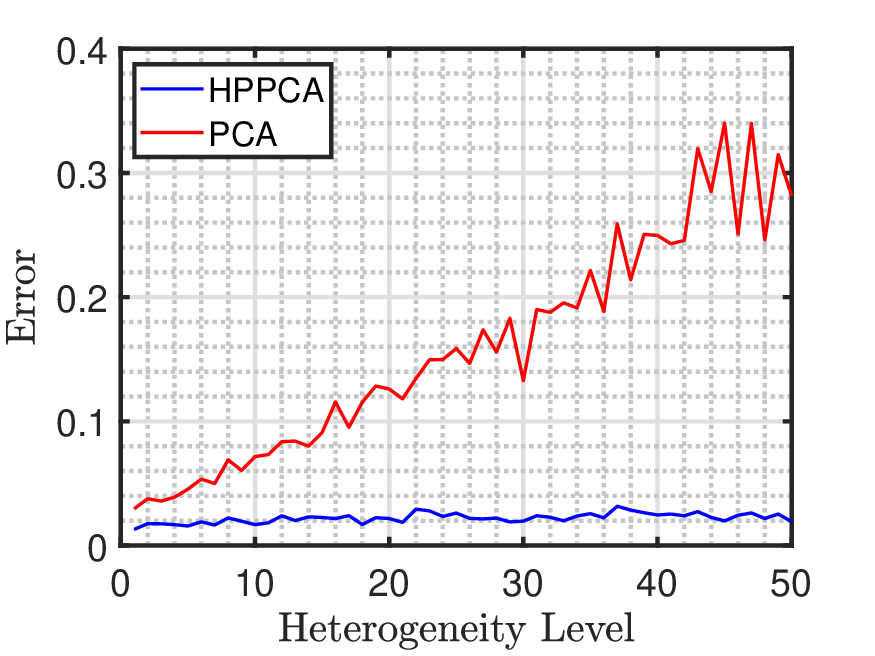}
			\caption{Robustness against heterogeneity.}  
			\label{fig:rob_heteo}
		\end{minipage} 
	\end{figure}

	\section{Concluding Remarks}
	We presented the generalized power method (GPM) to tackle the HPPCA problem and established its estimation performance guarantee as well as the associated linear rate. These theoretical results are obtained via a powerful error bound-based analysis framework. There are several future directions. First, numerical results show that GPM with random initialization works well but there is no theoretical guarantee. As such, providing its theoretical justifications would be intriguing. Second, numerical results show that GPM also performs well for HPPCA with sub-Gaussian noises. Hence, extending our theoretical results to tackle the general sub-Gaussian noise in the HPPCA setting is definitely worth exploring. Third, it would be exciting to establish the linear convergence result of GPM to a global maximizer of the HPPCA problem and extend strong theoretical results of GPM to more general settings (e.g., maximizing a sum of general heterogeneous quadratic forms with an orthogonality constraint). Last but not least, it is also interesting to obtain sharp entrywise estimation error bound by the ``leave-one-out'' technique in \cite{zhong2018near,chen2021spectral}.
	
	\section*{Acknowledgments}
	The authors thank Prof. Laura Balzano and Dr. Peng Wang for their helpful discussions.
	%The authors thank the anonymous reviewers for their valuable suggestions. qThis work is supported in part by funds from the National Science Foundation (NSF: \# 1636933 and \# 1920920).
	
	%\section*{Funding}
	%Funding information here.
	\bibliographystyle{plain}
	\bibliography{refs}
\end{document}